\documentclass{myaimsvar}
\usepackage{amsmath}
\usepackage{amssymb}
\usepackage{enumerate}
  \usepackage{graphics} 
  \usepackage{epsfig} 
 \usepackage[colorlinks=true]{hyperref}
\hypersetup{urlcolor=blue, citecolor=red}

\def\currentvolume{X}

     \def\DOI{10.3934/xx.xx.xx.xx}
     
\newtheorem{theorem}{Theorem}[section]
\newtheorem{corollary}{Corollary}

\newtheorem{lemma}[theorem]{Lemma}

\theoremstyle{definition}
\newtheorem{definition}[theorem]{Definition}

\newtheorem{remark}{Remark}

\newcommand{\ep}{\varepsilon}

\DeclareMathOperator{\conv}{conv} 
\DeclareMathOperator{\dju}{\dot\cup}
\DeclareMathOperator{\divergence}{div} 
\DeclareMathOperator{\essinf}{ess\,inf} 
\DeclareMathOperator{\grad}{\nabla} 
\DeclareMathOperator{\id}{Id}
\DeclareMathOperator{\interior}{int} 
\DeclareMathOperator{\nb}{nb}
\DeclareMathOperator{\ph}{ph} 
\DeclareMathOperator{\var}{var}

\newcommand{\aaref}[1]{{\rm (AA}-{\rm \ref{#1})}}
\newcommand{\aref}[1]{{\rm (A}-{\rm \ref{#1})}}
\newcommand{\daref}[1]{{\rm (DA}-{\rm \ref{#1})}}
\newcommand{\emaref}[1]{{\em (A}-{\em \ref{#1})}}
\newcommand{\emaaref}[1]{{\em (AA}-{\em \ref{#1})}}
\newcommand{\emdaref}[1]{{\em (DA}-{\em \ref{#1})}}

\newcommand{\dreg}{\partial_{{\rm reg}}}

\newcommand{\mch}{\mathcal H_{\nu,i}}
\newcommand{\myb}{b_{\nu,i}}
\newcommand{\partt}[1]{\frac{\partial #1}{\partial t}}
\newcommand{\q}{\quad}
\newcommand{\res}[1]{\hspace{-3pt}\upharpoonright_{#1}}
\newcommand{\tildeg}{{\tilde g}_{\nu,i}}
\newcommand{\tildeh}{{\tilde h}_{i}}
\newcommand{\vecmax}[1]{\left\|#1\right\|_{\max}}
\newcommand{\vecmin}[1]{\left\|#1\right\|_{\min}}

\renewcommand{\vec}[1]{\mathbf #1}

\newcounter{philipenumi}
\newcounter{philipenumiaa}
\newcounter{philipenumida}

\title[Implicit Scheme for Radiative Heat Transfer]
      {Fully Implicit Finite Volume Scheme for Transient Conductive-Radiative Heat Transfer: 
  Discrete Existence, Uniqueness, Maximum Principle}

\author[Peter Philip]{}


 \email{philip@math.lmu.de}

\begin{document}
\maketitle

\centerline{\scshape Peter Philip }
\medskip
{\footnotesize
 \centerline{Department of Mathematics}
   \centerline{Ludwig-Maximilians University (LMU) Munich}
   \centerline{Theresienstrasse 39, 80333 Munich, Germany}
} 

\bigskip

\begin{abstract}
  This article studies a fully implicit 
  finite volume scheme for transient nonlinear heat transport
  equations coupled by nonlocal interface conditions modeling diffuse-gray radiation
  between the surfaces of (both open and closed) cavities. The model is considered in
  three space dimensions; modifications for the axisymmetric case are
  indicated. Extending the results of \cite{KP-05}, where a similar, but
  not fully implicit, finite volume scheme was considered, 
  a discrete maximum principle is established, 
  yielding discrete $L^\infty$-$L^\infty$ a priori bounds 
  as well as a unique discrete solution to the finite volume scheme. 
\end{abstract}

\section{Introduction}

Modeling and numerical simulation of conductive-radiative heat
transfer has become a standard tool to aid and improve numerous
industrial processes such as crystal growth by the Czochralski 
method \cite{DNRWC-90,KLD+09} and by the physical vapor transport method \cite{KPS-04} 
to mention just two examples.

Heat transfer models including diffuse-gray radiative interactions between cavity surfaces
consist of nonlinear elliptic (stationary) or parabolic (transient)
PDE (heat equations), where a nonlocal coupling occurs due to the integral operator of the radiosity equation.
Recent papers regarding the mathematical theory of
existence, uniqueness, and regularity of weak solutions include
\cite{Amo-10vol165,Amo-10vol164,DKS+11,Dru-10}.
Discretization methods in the context of radiative heat transfer between 
surfaces of nonconvex cavities have been studied in \cite{KP-05,Tii-98}. 
In \cite{Tii-98}, a finite element
approximation is considered for a stationary conductive-radiative heat
transfer problem. In \cite{KP-05}, 
transient heat transport is treated and, in contrast to \cite{Tii-98},
heat conduction is also considered
inside closed cavities, with a jumping diffusion coefficient at the
interface. Moreover, the emissivity is allowed to depend on the
temperature (i.e.\ on the solution). The setting of \cite{KP-05} is also used in the following.

The finite volume scheme presented in the current paper is similar to
that of \cite{KP-05}, however, it differs from the scheme
in \cite{KP-05} in being {\em fully implicit}: The scheme in
\cite{KP-05} is an implicit scheme with the exception that the 
temperature-dependent emissivities are approximated explicitly, i.e.\
they are evaluated at the temperature of the previous time step. 
For the scheme presented in the current paper, the 
emissivities are also approximated implicitly, i.e.\ they are 
evaluated at the temperature of the current time step.
The advantage of the scheme in \cite{KP-05}, i.e.\ of the explicit
discretization of the emissivities, is a simpler implementation: The
discrete nonlinear system is solved via Newton's method, and approximating
the emissivities explicitly considerably simplifies
the computation of the derivative of the discrete nonlinear operator.
However, it is well-known that explicitly discretized terms in
transient heat transfer problems can impair the convergence properties 
of the scheme by convergence requiring additional smallness conditions on the
time step size $k$ depending on the size $h$ of the elements of the
space discretization.
Proving the convergence of the scheme of \cite{KP-05} as well as of the
scheme of the present paper to the weak solution of the corresponding
continuous problem does not seem easy and is still work in progress.
However, preliminary results indicate the convergence proof for the \cite{KP-05} scheme
requires an additional smallness condition of the form
$k\sim h^2$ on the time step size arising
from the explicitly discretized terms; and that this smallness condition can be avoided for the
fully implicit scheme of the present paper. The discrete existence results of 
\cite{KP-05} as well as of the present paper only require a smallness
condition of the form $k\sim h$.

Therefore, the goal of the present paper is to extend the results of 
\cite{KP-05}, i.e.\ the discrete maximum principle as well as
existence and uniqueness of a solution to the discrete scheme, to the fully implicit scheme. 
Both finite volume schemes lead to nonlinear and nonlocal systems of
equations, the solvability of which is not at all obvious. 
Using an additional regularity assumption for the emissivity function, 
namely local Lipschitzness, the main results can be proved by the 
same method as in \cite{KP-05}: The proof of the discrete maximum principle as well as
existence and uniqueness are based on the root problem 
with maximum principle \cite[Th.\ 4.1]{KP-05}. 

The paper is organized as follows: 
In Sec.\ \ref{sec:heat}, the governing equations of transient conductive heat
transfer are recalled, completed by nonlocal interface and boundary
conditions arising from the modeling of diffuse-gray radiation. 
Section \ref{sec:heat} also provides the precise mathematical
setting. The discrete scheme is stated in Sec.\ \ref{sec:fvm}, 
where the nonlocal radiation operators are
discretized in Sec.\ \ref{sec:nonlocdis}, also providing some important properties of
the resulting discrete nonlocal operators. 
The proof of the discrete maximum principle as well as existence and 
uniqueness of a discrete solution to the finite volume scheme
are the subject of Sec.\ \ref{sec:sol}. 
The main results are found in Th.\ \ref{th:sol} and its two corollaries. 

\section{Transient Heat Transport Including Conduction and Diffuse-Gray Radiation}\label{sec:heat}

\subsection{Transient Heat Equations}

Transient conductive-radiative heat transport is considered on a time-space cylinder
$[0,T]\times\overline\Omega$, where the space domain $\Omega\subseteq\mathbb R^3$
is assumed to consist of two parts $\Omega_{\rm s}$ and $\Omega_{\rm g}$,
$\Omega_{\rm s}$ representing an opaque solid
and $\Omega_{\rm g}$ representing a transparent gas.
More precisely, we assume:
\begin{enumerate}[({A}-1)]
\item\label{item:aomega} 
  $T\in\mathbb R^+$, $\overline\Omega={\overline\Omega}_{\rm s}\cup{\overline\Omega}_{\rm g}$, 
  $\Omega_{\rm s}\cap\,\Omega_{\rm g}=\emptyset$,
  and each of the sets $\Omega$,
  $\Omega_{\rm s}$, $\Omega_{\rm g}$,
  is a nonvoid, polyhedral, bounded, and open subset of $\mathbb R^3$. 
\item\label{eq:bndeqbnds} 
  $\Omega_{\rm g}$ is enclosed by $\Omega_{\rm s}$, 
  i.e.\ $\partial\Omega_{\rm s}=\partial\Omega\dju\partial\Omega_{\rm g}$, 
  where $\dju$ denotes a disjoint union. 
  Thus, 
  $\Sigma:=\partial\Omega_{\rm g}={\overline\Omega}_{\rm s}\cap{\overline\Omega}_{\rm g}$,
  and 
  $\partial\Omega=\partial\Omega_{\rm s}\setminus\Sigma$ (see Fig.\ \ref{fig:domain}).
  \setcounter{philipenumi}{\value{enumi}}
\end{enumerate}
Heat conduction is considered throughout $\overline\Omega$. 
Nonlocal radiative heat transport is considered between points on the
surface $\Sigma$ of $\Omega_{\rm g}$ as well as between points on the
surfaces of open cavities 
(such as $O_1$ and $O_2$ in Fig.\ \ref{fig:domain}). However, 
to avoid introducing additional
boundary conditions, open cavities are not part of 
$\overline\Omega$, i.e.\ heat conduction is {\em not} considered in
open cavities (see Sec.\ \ref{sec:bnd} below for details).

\begin{figure}[htb!] 
  \begin{center}
\begin{picture}(0,0)%
\includegraphics{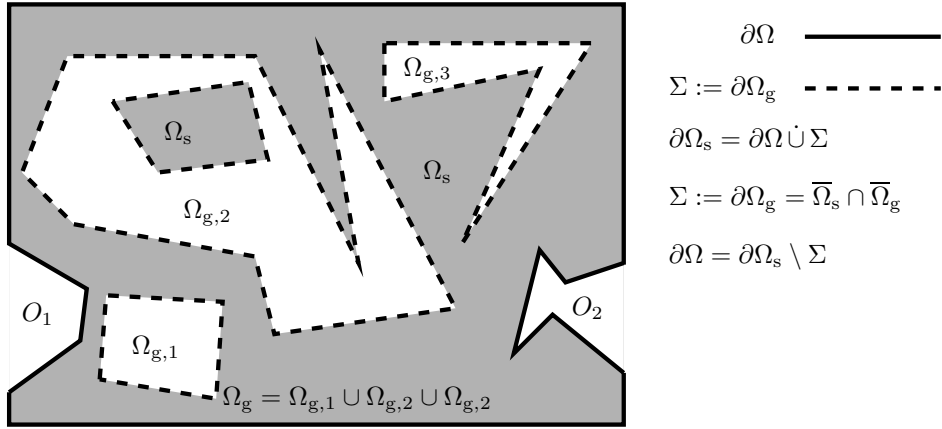}%
\end{picture}%
\setlength{\unitlength}{1782sp}%
\begingroup\makeatletter\ifx\SetFigFont\undefined%
\gdef\SetFigFont#1#2#3#4#5{%
  \reset@font\fontsize{#1}{#2pt}%
  \fontfamily{#3}\fontseries{#4}\fontshape{#5}%
  \selectfont}%
\fi\endgroup%
\begin{picture}(13068,5958)(1297,-6415)
\put(3511,-2401){\makebox(0,0)[lb]{\smash{{\SetFigFont{10}{12.0}{\rmdefault}{\mddefault}{\updefault}{\color[rgb]{0,0,0}$\Omega_{\rm s}$}%
}}}}
\put(7111,-2941){\makebox(0,0)[lb]{\smash{{\SetFigFont{10}{12.0}{\rmdefault}{\mddefault}{\updefault}{\color[rgb]{0,0,0}$\Omega_{\rm s}$}%
}}}}
\put(3781,-3481){\makebox(0,0)[lb]{\smash{{\SetFigFont{10}{12.0}{\rmdefault}{\mddefault}{\updefault}{\color[rgb]{0,0,0}$\Omega_{{\rm g},2}$}%
}}}}
\put(6841,-1501){\makebox(0,0)[lb]{\smash{{\SetFigFont{10}{12.0}{\rmdefault}{\mddefault}{\updefault}{\color[rgb]{0,0,0}$\Omega_{{\rm g},3}$}%
}}}}
\put(4321,-6091){\makebox(0,0)[lb]{\smash{{\SetFigFont{10}{12.0}{\rmdefault}{\mddefault}{\updefault}{\color[rgb]{0,0,0}$\Omega_{\rm g}=\Omega_{{\rm g},1}\cup\Omega_{{\rm g},2}\cup\Omega_{{\rm g},2}$}%
}}}}
\put(3061,-5371){\makebox(0,0)[lb]{\smash{{\SetFigFont{10}{12.0}{\rmdefault}{\mddefault}{\updefault}{\color[rgb]{0,0,0}$\Omega_{{\rm g},1}$}%
}}}}
\put(9181,-4831){\makebox(0,0)[lb]{\smash{{\SetFigFont{10}{12.0}{\rmdefault}{\mddefault}{\updefault}{\color[rgb]{0,0,0}$O_2$}%
}}}}
\put(1531,-4921){\makebox(0,0)[lb]{\smash{{\SetFigFont{10}{12.0}{\rmdefault}{\mddefault}{\updefault}{\color[rgb]{0,0,0}$O_1$}%
}}}}
\put(10531,-2491){\makebox(0,0)[lb]{\smash{{\SetFigFont{10}{12.0}{\rmdefault}{\mddefault}{\updefault}{\color[rgb]{0,0,0}$\partial\Omega_{\rm s}=\partial\Omega\dju\Sigma$}%
}}}}
\put(10531,-3301){\makebox(0,0)[lb]{\smash{{\SetFigFont{10}{12.0}{\rmdefault}{\mddefault}{\updefault}{\color[rgb]{0,0,0}$\Sigma:=\partial\Omega_{\rm g}={\overline{\Omega}}_{\rm s}\cap{\overline{\Omega}}_{\rm g}$}%
}}}}
\put(10531,-4111){\makebox(0,0)[lb]{\smash{{\SetFigFont{10}{12.0}{\rmdefault}{\mddefault}{\updefault}{\color[rgb]{0,0,0}$\partial\Omega=\partial\Omega_{\rm s}\setminus\Sigma$}%
}}}}
\put(11521,-1051){\makebox(0,0)[lb]{\smash{{\SetFigFont{10}{12.0}{\rmdefault}{\mddefault}{\updefault}{\color[rgb]{0,0,0}$\partial\Omega$}%
}}}}
\put(10531,-1771){\makebox(0,0)[lb]{\smash{{\SetFigFont{10}{12.0}{\rmdefault}{\mddefault}{\updefault}{\color[rgb]{0,0,0}$\Sigma:=\partial\Omega_{\rm g}$}%
}}}}
\end{picture}%
  \end{center}
  \caption{Possible shape of a 2-dimensional section through the
    3-dimensional domain 
    $\overline\Omega=\overline\Omega_{\rm s}\cup\overline\Omega_{\rm g}$ with open 
    cavities $O_1$ and $O_2$. Note that, according to \aref{eq:bndeqbnds},
    $\Omega_{\rm g}$ is engulfed by $\Omega_{\rm s}$, which
    can not be seen in the 2-dimensional section.}\label{fig:domain}
\end{figure}

\noindent Transient heat conduction is described by 
\begin{equation}
  \label{eq:heatconsol}
  \partt{\ep_m(\theta)}
  -\divergence(\kappa_m\grad\theta)  = f_m(t,x)\q
  \text{in $]0,T[\times\Omega_m$\q\q ($m\in\{{\rm s,g}\}$)},
\end{equation}
where $\theta(t,x)\in\mathbb R^+_0$ represents absolute temperature, 
depending on the time coordinate $t$ and on the space
coordinate $x$; the continuous, strictly increasing, 
nonnegative functions $\ep_m\in C(\mathbb R^+_0,\mathbb R^+_0)$ 
represent the internal energy in the solid and in the gas, respectively, 
$\kappa_m\in\mathbb R^+_0$ represent the
thermal conductivity in solid and gas, respectively, assumed constant for simplicity, 
and $f_m$ represent heat sources due to some heating
mechanism. In practice, for many heating mechanisms such as
induction or resistance heating, one has $f_{\rm g}=0$. 

Throughout this paper, \aref{item:ae} -- \aref{item:af} are assumed, where:
\begin{enumerate}[({A}-1)]
  \setcounter{enumi}{\thephilipenumi}
\item\label{item:ae} For $m\in\{{\rm s,g}\}$, 
  $\ep_m:\,\mathbb R^+_0\longrightarrow\mathbb R^+_0$ is continuous and
  at least of linear growth, i.e.\ there is $C_{\ep}\in\mathbb R^+$ such that 
  \begin{equation*}
    \ep_m(\theta_2)\geq (\theta_2-\theta_1)\,C_{\ep}+\ep_m(\theta_1)\q\q
    (\theta_2\geq\theta_1\geq 0).
  \end{equation*}    
\item\label{item:ak} For $m\in\{{\rm s,g}\}$: $\kappa_m\in\mathbb R^+_0$.
\item\label{item:af} For $m\in\{{\rm s,g}\}$: $f_m\in L^\infty\big(]0,T[\times\Omega_m\big)$,
  $f_m\geq 0$ a.e. 
  \setcounter{philipenumi}{\value{enumi}}
\end{enumerate}

\subsection{Nonlocal Interface Conditions}\label{sec:if}

Continuity of the heat flux on the interface $\Sigma$ between solid and gas, 
where one needs to account for
radiosity $R$ and for irradiation $J$,
yields the following interface condition
for \eqref{eq:heatconsol}:
\begin{equation}
  \label{eq:intcon}
  \big(\kappa_{\rm g}\grad\theta\big)\hspace{-3pt}\upharpoonright_{\overline\Omega_{\rm g}}\cdot\,\vec n_{\rm g} 
  +R(\theta)-J(\theta)
  =\big(\kappa_{\rm s}\grad\theta\big)\hspace{-3pt}\upharpoonright_{\overline\Omega_{\rm s}}\cdot\,\vec n_{\rm g}\quad
  \text{on $]0,T[\times\Sigma$}.
\end{equation}
Here, $\vec n_{\rm g}$
denotes the unit normal vector pointing from gas to solid and $\hspace{-3pt}\upharpoonright_{}$ denotes restriction
(or trace).

As the solid is assumed opaque,
$R(\theta)$ and $J(\theta)$ are computed
according to the net radiation model for
diffuse-gray surfaces, i.e.\ reflection and emittance 
are taken to be independent of the angle of 
incidence and independent of the wavelength.
At each point of the surface $\Sigma$,
the radiosity is the sum of the emitted radiation 
$E(\theta)$ and of the
reflected radiation $J_{\rm r}(\theta)$:
\begin{equation}
  \label{eq:rej}
  R=E+J_{\rm r}\;
  \text{ on $\Sigma$}.
\end{equation}
According to the Stefan-Boltzmann law, 
\begin{equation}
  \label{eq:sblaw}
  E(\theta)=\sigma\,\epsilon(\theta)\,\theta^4\;
  \text{ on $\Sigma$},
\end{equation}
where $\sigma$ represents the Boltzmann radiation constant,
and $\epsilon$ represents the potentially tem\-per\-a\-ture-dependent 
emissivity of the solid surface. It is assumed that:
\begin{enumerate}[({A}-1)]
  \setcounter{enumi}{\thephilipenumi}
\item\label{item:em} 
  $\sigma\in\mathbb R^+$, 
  $\epsilon:\,\mathbb R^+_0\longrightarrow]0,1]$
  is locally Lipschitz continuous, i.e., for each $r\in\mathbb R^+_0$:
  \begin{equation*}
    L_{\epsilon,r}
    :=
    \sup
    \left\{
      \frac{|\epsilon(\theta_1)-\epsilon(\theta_2)|}{|\theta_1-\theta_2|}:\,
      (\theta_1,\theta_2)\in[0,r]^2,\,
      \theta_1\not=\theta_2
    \right\}
    \in\mathbb R^+_0.
  \end{equation*}
  \setcounter{philipenumi}{\value{enumi}}
\end{enumerate}
\begin{remark}
  Condition \aref{item:em} is stronger than \cite[(A-6)]{KP-05}, 
  where $\epsilon$ was only assumed to be continuous, but not
  necessarily locally Lipschitz. For the proof of \cite[Lem.\ 3.2(c)]{KP-05},
  i.e.\ for the proof of the local Lipschitz continuity of the maps
  $V_{\Gamma,\alpha}(\tilde{\vec u},\cdot)$ and 
  $V_{\Sigma,\alpha}(\tilde{\vec u},\cdot)$, respectively,
  due to the explicit discretization of $\epsilon$, the regularity of
  $\epsilon$ was not an issue. However, here, where $\epsilon$ is
  discretized implicitly, local Lipschitz continuity of $\epsilon$ is
  necessary to proof the respective local Lipschitz continuity of 
  $V_{\Gamma,\alpha}$ and $V_{\Sigma,\alpha}$ in 
  Lem.\ \ref{lem:nonnegrad}\eqref{item:nonnegradb} below
  (in general, if $\epsilon$ is only continuous, then local Lipschitz
  continuity of $V_{\Gamma,\alpha}$ and $V_{\Sigma,\alpha}$ can not be expected).
\end{remark}
\begin{remark}
  \label{rem:epsminr}
  For each $r\in\mathbb R^+$, 
  the continuous map $\epsilon$ attains its minimum,
  denoted by $\epsilon_{\min,r}$, on the compact set $[0,r]$.
  Since $\epsilon>0$ according to \aref{item:em}, it follows that
  $\epsilon_{\min,r}>0$ for each $r\in\mathbb R^+$.
\end{remark}

Using the presumed opaqueness together with Kirchhoff's law yields 
\begin{equation}
  \label{eq:jrefwref}
  J_{\rm r}=(1-\epsilon)\, J. 
\end{equation}
Due to diffuseness, the irradiation can be calculated as
\begin{equation}
  \label{eq:jandk}
  J(\theta)=K(R(\theta)),
\end{equation}
using the nonlocal integral radiation operator $K$ defined by
\begin{align}
  \label{eq:K}
  K(\rho)(x)
  &:=\int_{\Sigma}\Lambda(x,y)\,\omega(x,y)\,\rho(y)\, {\mathrm{d}} y \,
  \quad\text{for a.e.\ $x\in\Sigma$},\\[1mm]
  \omega(x,y)&:=
  \dfrac{ \big(\vec{n}_{\rm s}(y) \cdot ( x - y)\big) \, \big(\vec{n}_{\rm s}(x) \cdot ( y - x)\big) }
  { \pi\big( (y - x ) \cdot (y - x )\big)^2 }
  \quad\text{for a.e.\ $(x,y)\in\Sigma\times\Sigma$},
  \label{eq:viewfacs}\\[1mm]
  \Lambda(x,y)&:=
  \begin{cases}
    0&\text{if $\Sigma\,\cap\,]x,y[\not=\emptyset$},\\
    1&\text{if $\Sigma\,\cap\,]x,y[=\emptyset$}
  \end{cases}
  \quad\text{for each $(x,y)\in\Sigma\times\Sigma$},
\end{align}
where $\omega$ is called view factor,
$\Lambda$ is called visibility factor (being $1$ if, and only if,
$x$ and $y$ are mutually visible),
and $\vec{n}_{\rm s}$ denotes the outer unit normal to the
solid domain $\Omega_{\rm s}$, 
existing almost everywhere on the Lipschitz interface $\Sigma$.
The following Th.\ \ref{th:kop} summarizes properties of
$\omega$, $\Lambda$, and $K$, relevant to our considerations.
\begin{theorem}
  \label{th:kop}
  Assume \aref{item:aomega} and \aref{eq:bndeqbnds}.
  \begin{enumerate}[\bf(a)]
  \item\label{item:thkopa}
    The kernel $\Lambda\omega$ of $K$ is almost everywhere
    nonnegative (actually positive for $\Lambda(x,y)$ $=1$), symmetric, and
    $\Lambda(x,\cdot)\,\omega(x,\cdot)$ is in $L^1(\Sigma)$ with
    \begin{equation}
      \label{eq:itemthkopa}
      \int_{\Sigma}\Lambda(x,y)\,\omega(x,y)\, {\mathrm{d}} y \,=1
      \quad\text{for a.e.\ $x\in\Sigma$}.
    \end{equation}
  \item\label{item:thkopb}
    For each $1\leq p\leq\infty$, the operator
    $K:\,L^p(\Sigma)\longrightarrow L^p(\Sigma)$ given by \eqref{eq:K}
    is well-defined, linear, bounded, and positive with $\|K\|=1$.
  \end{enumerate}
\end{theorem}
\begin{proof}
  See \cite[Lem.\ 1]{Tii-97ejam}
  and \cite[Lem.\ 2]{Tii-97mmias}.
\end{proof}
Combining \eqref{eq:rej} through \eqref{eq:jandk} provides the
so-called radiosity equation for $R$:
\begin{equation}
  \label{eq:radiositys}
  \big(\id-(1-\epsilon(\theta))K\big)(R)
  = \sigma\,\epsilon(\theta)\,\theta^4,
\end{equation}
where $\id$ denotes the identity operator. 
The following Th.\ \ref{th:inv} allows to solve 
\eqref{eq:radiositys} for $R$. 
\begin{theorem}
  \label{th:inv}
  Assume \aref{item:aomega}, \aref{eq:bndeqbnds}, \aref{item:em}.
  \begin{enumerate}[\bf(a)]
  \item\label{item:thinva}
    Let $p\in[1,\infty]$. Then, for each $\theta\in L^1(\Sigma)$,
    the operator $\id-(1-\epsilon(\theta))K$ has an inverse in 
    the Banach space $\mathcal L(L^p(\Sigma),L^p(\Sigma))$ of bounded linear operators.
  \item\label{item:thinvb}
    Let $p\in[1,\infty]$. For each $\theta\in L^{4p}(\Sigma)$,
    the radiosity equation \eqref{eq:radiositys} has the unique solution
    $R(\theta)=\big(\id-(1-\epsilon(\theta))K\big)^{-1}\big(\sigma\,\epsilon(\theta)\,\theta^4\big)\in L^p(\Sigma)$.
  \end{enumerate}
\end{theorem}
\begin{proof}
  Part \eqref{item:thinva} is given by \cite[Th.\ 5]{DP-11},
  since $\Sigma$ is assumed polyhedral by \aref{item:aomega};
  \eqref{item:thinvb} follows from \eqref{item:thinva}, as $\sigma>0$ and $\epsilon(\theta)\in L^\infty(\Sigma)$ by \aref{item:em}.
\end{proof}
With the computation
\begin{eqnarray}
  R(\theta)-J(\theta)
  &\overset{\eqref{eq:rej}}=&
  E(\theta)+J_{\rm r}(\theta)-J(\theta)
  \overset{\text{\eqref{eq:sblaw},\eqref{eq:jrefwref}}}=
  \sigma\,\epsilon(\theta)\,\theta^4-\epsilon(\theta)\,J(\theta)\nonumber\\
  &\overset{\eqref{eq:jandk}}=&
  -\epsilon(\theta)\,\big(K(R(\theta))-\sigma\,\theta^4\big),
  \label{eq:minrplj}
\end{eqnarray}
\eqref{eq:intcon} becomes
\begin{equation}
  \label{eq:intconop}
  (\kappa_{\rm g}\grad\theta)\res{\overline\Omega_{\rm g}}\cdot\,\vec n_{\rm g} 
  -\epsilon(\theta)\,\big(K(R(\theta))-\sigma\,\theta^4\big) 
  =(\kappa_{\rm s}\grad\theta)\res{\overline\Omega_{\rm s}}\cdot\,\vec n_{\rm g}\q
  \text{on $]0,T[\times\Sigma$},
\end{equation}
where $R(\theta)$ is given by Th.\ \ref{th:inv}\eqref{item:thinvb}.

\subsection{Nonlocal Outer Boundary Conditions}\label{sec:bnd}

\begin{definition}
  \label{def:o}
  Let $\conv(\overline\Omega)$ denote the closed convex hull of $\Omega$, 
  and define $O:=\interior(\conv(\overline\Omega))\setminus\overline \Omega$, 
  $\Gamma_\Omega:=\overline\Omega\cap\overline O$, 
  $\Gamma:=\partial O$, and
  $\Gamma_{\rm ph}
  :=\partial\conv(\overline\Omega)\cap\partial O$.
  Then $(\Gamma_\Omega,\Gamma_{\rm ph})$ forms a partition of $\Gamma$. 
  The set $O$ is the domain of the open
  radiation region (e.g., one has $O=O_1\cup O_2$ in 
  Figures \ref{fig:domain} and \ref{fig:openRad}).   
\end{definition}
\bigskip

The condition on the interface between $\Omega$ and 
the open radiation region $O$ reads
\begin{equation}
  \label{eq:intconopen}
  \kappa_{\rm s}\grad\theta\cdot\,\vec n_{\rm s} 
  +R_\Gamma(\theta)-J_\Gamma(\theta)
  =0\q
  \text{on $]0,T[\times\Gamma_{\Omega}$}
\end{equation}
in analogy with \eqref{eq:intcon}, where $\vec n_{\rm s}$ is the outer
unit normal vector to the solid. To allow for radiative interactions 
between surfaces of open cavities and the ambient environment, including
reflections at the cavity's surfaces, the set $\Gamma_{\rm ph}$ as
defined above, is used as a black body phantom 
closure (see Fig.\ \ref{fig:openRad}), 
emitting radiation at an external temperature $\theta_{\rm ext}$,
\begin{enumerate}[({A}-1)]
  \setcounter{enumi}{\thephilipenumi}
\item\label{item:thetar} 
  $\theta_{\rm ext}\in\mathbb R^+$. 
  \setcounter{philipenumi}{\value{enumi}}
\end{enumerate}
Thus, $\epsilon\equiv 1$
on $\Gamma_{\rm ph}$, implying
\begin{equation}
  \label{eq:ronphan}
  R_\Gamma(\theta)(x)=\sigma\,\theta_{\rm ext}^4\q (x\in\Gamma_{\rm ph}).
\end{equation}
Here and in the following, it is
assumed that the apparatus is exposed to a
black body environment (e.g.\ a large isothermal
room) radiating at $\theta_{\rm ext}$. 
A relation analogous to \eqref{eq:minrplj} holds on 
$\Gamma_{\Omega}$, and using it in \eqref{eq:intconopen} yields 
\begin{equation}
  \label{eq:opencav}
  \kappa_{\rm s}\grad\theta\cdot\,\vec n_{\rm s}
  -\epsilon(\theta)\,\big(K_\Gamma(R_\Gamma(\theta))-\sigma\,\theta^4\big) 
  =0\q
  \text{on $]0,T[\times\Gamma_{\Omega}$},
\end{equation}
where $K_\Gamma$ is defined analogous to $K$
in \eqref{eq:K}, except that the integration is carried out over 
$\Gamma$ instead of over $\Sigma$. 

\begin{figure}[htb!] 
  \begin{center}
\begin{picture}(0,0)%
\includegraphics{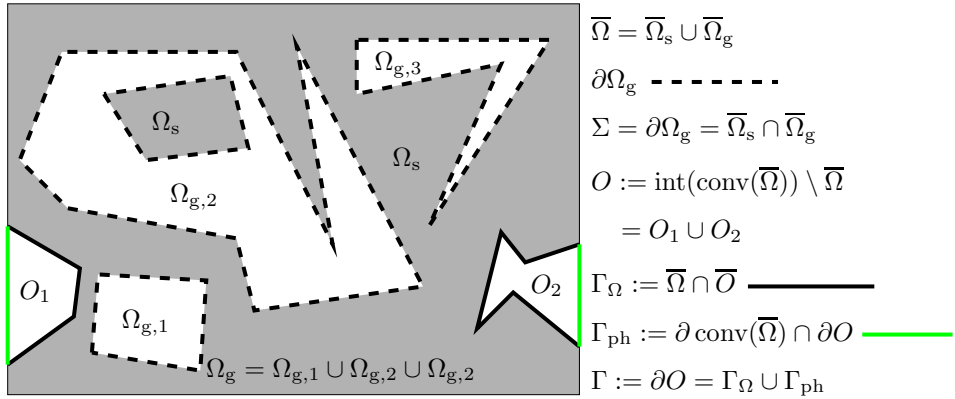}%
\end{picture}%
\setlength{\unitlength}{1657sp}%
\begingroup\makeatletter\ifx\SetFigFont\undefined%
\gdef\SetFigFont#1#2#3#4#5{%
  \reset@font\fontsize{#1}{#2pt}%
  \fontfamily{#3}\fontseries{#4}\fontshape{#5}%
  \selectfont}%
\fi\endgroup%
\begin{picture}(14238,5904)(1297,-6403)
\put(3511,-2401){\makebox(0,0)[lb]{\smash{{\SetFigFont{10}{12.0}{\rmdefault}{\mddefault}{\updefault}{\color[rgb]{0,0,0}$\Omega_{\rm s}$}%
}}}}
\put(7111,-2941){\makebox(0,0)[lb]{\smash{{\SetFigFont{10}{12.0}{\rmdefault}{\mddefault}{\updefault}{\color[rgb]{0,0,0}$\Omega_{\rm s}$}%
}}}}
\put(3781,-3481){\makebox(0,0)[lb]{\smash{{\SetFigFont{10}{12.0}{\rmdefault}{\mddefault}{\updefault}{\color[rgb]{0,0,0}$\Omega_{{\rm g},2}$}%
}}}}
\put(6841,-1501){\makebox(0,0)[lb]{\smash{{\SetFigFont{10}{12.0}{\rmdefault}{\mddefault}{\updefault}{\color[rgb]{0,0,0}$\Omega_{{\rm g},3}$}%
}}}}
\put(4321,-6091){\makebox(0,0)[lb]{\smash{{\SetFigFont{10}{12.0}{\rmdefault}{\mddefault}{\updefault}{\color[rgb]{0,0,0}$\Omega_{\rm g}=\Omega_{{\rm g},1}\cup\Omega_{{\rm g},2}\cup\Omega_{{\rm g},2}$}%
}}}}
\put(3061,-5371){\makebox(0,0)[lb]{\smash{{\SetFigFont{10}{12.0}{\rmdefault}{\mddefault}{\updefault}{\color[rgb]{0,0,0}$\Omega_{{\rm g},1}$}%
}}}}
\put(9181,-4831){\makebox(0,0)[lb]{\smash{{\SetFigFont{10}{12.0}{\rmdefault}{\mddefault}{\updefault}{\color[rgb]{0,0,0}$O_2$}%
}}}}
\put(1531,-4921){\makebox(0,0)[lb]{\smash{{\SetFigFont{10}{12.0}{\rmdefault}{\mddefault}{\updefault}{\color[rgb]{0,0,0}$O_1$}%
}}}}
\put(10081,-1051){\makebox(0,0)[lb]{\smash{{\SetFigFont{10}{12.0}{\rmdefault}{\mddefault}{\updefault}{\color[rgb]{0,0,0}$\overline{\Omega}={\overline{\Omega}}_{\rm s}\cup{\overline{\Omega}}_{\rm g}$}%
}}}}
\put(10081,-1771){\makebox(0,0)[lb]{\smash{{\SetFigFont{10}{12.0}{\rmdefault}{\mddefault}{\updefault}{\color[rgb]{0,0,0}$\partial\Omega_{\rm g}$}%
}}}}
\put(10081,-2491){\makebox(0,0)[lb]{\smash{{\SetFigFont{10}{12.0}{\rmdefault}{\mddefault}{\updefault}{\color[rgb]{0,0,0}$\Sigma=\partial\Omega_{\rm g}={\overline{\Omega}}_{\rm s}\cap{\overline{\Omega}}_{\rm g}$}%
}}}}
\put(10081,-3301){\makebox(0,0)[lb]{\smash{{\SetFigFont{10}{12.0}{\rmdefault}{\mddefault}{\updefault}{\color[rgb]{0,0,0}$O:=\interior(\conv(\overline{\Omega}))\setminus\overline{\Omega}$}%
}}}}
\put(10081,-5551){\makebox(0,0)[lb]{\smash{{\SetFigFont{10}{12.0}{\rmdefault}{\mddefault}{\updefault}{\color[rgb]{0,0,0}$\Gamma_{\rm ph}:=\partial\conv(\overline{\Omega})\cap\partial O$}%
}}}}
\put(10081,-6271){\makebox(0,0)[lb]{\smash{{\SetFigFont{10}{12.0}{\rmdefault}{\mddefault}{\updefault}{\color[rgb]{0,0,0}$\Gamma:=\partial O=\Gamma_\Omega\cup\Gamma_{\rm ph}$}%
}}}}
\put(10081,-4831){\makebox(0,0)[lb]{\smash{{\SetFigFont{10}{12.0}{\rmdefault}{\mddefault}{\updefault}{\color[rgb]{0,0,0}$\Gamma_\Omega:=\overline{\Omega}\cap\overline{O}$}%
}}}}
\put(10531,-4021){\makebox(0,0)[lb]{\smash{{\SetFigFont{10}{12.0}{\rmdefault}{\mddefault}{\updefault}{\color[rgb]{0,0,0}$=O_1\cup O_2$}%
}}}}
\end{picture}%
  \end{center}
  \caption{For the domain of Fig.\ \ref{fig:domain}, the surfaces of 
    radiation regions are shown. The open radiation regions 
    $O_1$ and $O_2$ are artificially closed by the phantom closure
    $\Gamma_{\rm ph}$. 
    As in Fig.\ \ref{fig:domain}, Fig.\ \ref{fig:openRad} depicts a 
    2-dimensional section through the $3$-dimensional domain.}\label{fig:openRad}
\end{figure}

On parts of $\partial\Omega$ that do not interact radiatively with
other parts of the apparatus, i.e.\ on 
$\partial\Omega\setminus\Gamma_{\Omega}$, 
the Stefan-Boltzmann law provides the outer boundary condition 
\begin{equation}
  \label{eq:outerrad}
  \kappa_{\rm s}\grad\theta\cdot\,\vec n_{\rm s}
  -\sigma\,\epsilon(\theta)\,
  (\theta_{\rm ext}^4-\,\theta^4)=0\q
  \text{on $]0,T[\times\big(\partial\Omega\setminus\Gamma_{\Omega}\big)$}.
\end{equation}

\subsection{Initial Condition}

The initial condition reads $\theta(0,x)=\theta_{\rm init}(x)$, $x\in\Omega$,
where it is assumed that 
\begin{enumerate}[({A}-1)]
  \setcounter{enumi}{\thephilipenumi}
\item\label{item:initial}
  $\theta_{\rm init}\in L^\infty(\Omega,\mathbb R^+)$.
  \setcounter{philipenumi}{\value{enumi}}
\end{enumerate}

\section{The Finite Volume Scheme}\label{sec:fvm}

In this section, the finite volume scheme is formulated. For a
detailed derivation of the scheme see \cite[Sec.\ 3]{KP-05},
except that, where the emissivities are evaluated at the temperature
of the previous time step in \cite[Sec.\ 3]{KP-05}, one has to use the 
temperature of the current time step to obtain the fully implicit
finite volume scheme considered in the following.

\subsection{Discretization of Time and Space Domain}

A discretization of the time domain $[0,T]$ is given by an increasing
finite sequence $0=t_0<\dots<t_N=T$, $N\in\mathbb N$. 
The notation $k_\nu:=t_\nu-t_{\nu-1}$ will be used for the time steps. 

An admissible discretization of the space domain $\Omega$ is given by a finite family 
$\mathcal T:=(\omega_i)_{i\in I}$ of subsets of $\Omega$ satisfying a
number of assumptions, subsequently denoted by (DA-$*$).
\begin{enumerate}[({DA}-1)]
\item\label{item:spdis} 
  $\mathcal T=(\omega_i)_{i\in I}$ forms a partition of $\Omega$,
  and, for each $i\in I$, 
  $\omega_i$ is a nonvoid, polyhedral, connected, and open subset of $\Omega$. 
  \setcounter{philipenumida}{\value{enumi}}
\end{enumerate}
From $\mathcal T$, one can define discretizations of $\Omega_{\rm s}$ 
and $\Omega_{\rm g}$: For $m\in\{{\rm s,g}\}$ and $i\in I$, let 
\begin{align}
  \label{eq:omegakjdef}
  \omega_{m,i}:=\omega_i\cap\Omega_m,\q
  I_m:=\big\{j\in I:\,\omega_{m,j}\not=\emptyset\big\},\q
  \mathcal T_m:=(\omega_{m,i})_{i\in I_m}.
\end{align}

\begin{enumerate}[({DA}-1)]
  \setcounter{enumi}{\thephilipenumida}
\item\label{eq:nontangentdef} 
  For each $i\in I$: 
  $\dreg\omega_{{\rm s},i}\cap\Sigma
  =\dreg\omega_{{\rm g},i}\cap\Sigma$,
  where $\dreg$ denotes the regular boundary of a polyhedral set, i.e.\ the
  parts of the boundary, where a unique outer unit normal vector
  exists, $\dreg\emptyset:=\emptyset$. 
  \setcounter{philipenumida}{\value{enumi}}
\end{enumerate}

The boundary of each control volume $\omega_{m,i}$ can be
decomposed according to 
\begin{subequations}
  \label{subeq:bnddecom}
  \begin{equation}
    \label{eq:bnddecom}
    \partial\omega_{m,i}
    =\big(\partial\omega_{m,i}\cap\,\Omega_m\big)
    \;\cup\;
    \big(\partial\omega_{m,i}\cap\partial\Omega\big)
    \;\cup\;
    \big(\partial\omega_{m,i}
    \cap\Sigma\big).
  \end{equation}
  Recalling \aref{item:aomega}, \aref{eq:bndeqbnds}, and Def.\ \ref{def:o}, 
  outer boundary sets are decomposed further into 
  \begin{equation}
    \label{eq:bnddecomouter}
    \partial\omega_{{\rm s},i}\cap\partial\Omega
    =\big(\partial\omega_{{\rm s},i}\cap\Gamma_\Omega\big)
    \cup\big(\partial\omega_{{\rm s},i}\cap(\partial\Omega\setminus\Gamma_\Omega)\big),
  \end{equation} 
  whereas $\partial\omega_{{\rm g},i}\cap\partial\Omega=\emptyset$. 
\end{subequations}

Associate a discretization point $x_i\in\overline\omega_i$ with 
each control volume $\omega_i$ (cf.\ \cite{FL-00}). Then $\theta_{\nu,i}$ can be
interpreted as $\theta(t_\nu,x_i)$. Moreover, the 
discretization makes use of 
regularity assumptions concerning the partition $(\omega_i)_{i\in I}$
that can be expressed in terms of the $x_i$:
\begin{enumerate}[({DA}-1)]
  \setcounter{enumi}{\thephilipenumida}
\item\label{eq:xkinallj} 
  For each $m\in\{{\rm s,g}\}$, $i\in I_m$: 
  $x_i\in\overline\omega_{m,i}$. In particular,
  if $\omega_{{\rm s},i}\not=\emptyset$ and 
  $\omega_{{\rm g},i}\not=\emptyset$, then 
  $x_i\in\overline\omega_{{\rm s},i}\cap\overline\omega_{{\rm g},i}$.
\item\label{eq:xkinbnd}
  For each $i\in I$, the following holds: 
  If $\lambda_{2}(\overline\omega_i\cap\Gamma_\Omega)\not=0$,
  then $x_i\in\overline\omega_i\cap\Gamma_\Omega$;
  and, if $\lambda_{2}\big(\overline\omega_i\cap(\partial\Omega\setminus\Gamma_\Omega)\big)\not=0$,
  then $x_i\in\overline\omega_i\cap\overline{\partial\Omega\setminus\Gamma_\Omega}$.
  \setcounter{philipenumida}{\value{enumi}}
\end{enumerate}
\begin{remark}
  \label{rem:admomega}
  By \aref{eq:bndeqbnds}, \daref{eq:nontangentdef},
  \daref{eq:xkinallj}, \daref{eq:xkinbnd}, 
  $\overline\omega_i$ can {\em not} have 2-dimensional
  intersections with both $\partial\Omega$ and $\Sigma$. 
\end{remark}

Introducing the sets
\begin{subequations}
  \label{subeq:nbs}
  \begin{align}
    \label{eq:nbsm}
    \nb_m(i)&:=\{j\in I_m\setminus\{i\}:\,\lambda_{2}(\partial\omega_{m,i}\cap\partial\omega_{m,j})\not=0\},\\
    \nb(i)&:=\{j\in I\setminus\{i\}:\,\lambda_{2}(\partial\omega_i\cap\partial\omega_j)\not=0\},
    \label{eq:nbs}
  \end{align}
\end{subequations}
$\partial\omega_{m,i}\cap\,\Omega_m$
is partitioned further: 
\begin{equation}
  \label{eq:partintonbs}
  \partial\omega_{m,i}\cap\,\Omega_m
  =\bigcup_{j\in\nb_m(i)}\partial\omega_{m,i}\cap\partial\omega_{m,j},
\end{equation}
where it is assumed that: 
\begin{enumerate}[({DA}-1)]
  \setcounter{enumi}{\thephilipenumida}
\item\label{eq:poldisf} 
  For each $i\in I$, $j\in\nb(i)$:
  $x_i\not=x_j$ and 
  $\frac{x_j-x_i}
  {\left\|x_i-x_j\right\|_2}
  =\vec n_{\omega_i}\res{\partial\omega_i\cap\partial\omega_j}$,
  where $\|\cdot\|_2$ denotes Euclidean distance, 
  and $\vec n_{\omega_i}\res{\partial\omega_i\cap\partial\omega_j}$ is the restriction
  of the normal vector $\vec n_{\omega_i}$ to the interface 
  $\partial\omega_i\cap\partial\omega_j$. 
  Thus, the line segment joining neighboring vertices $x_i$ and $x_j$
  is always perpendicular to 
  $\partial\omega_i\cap\partial\omega_j$.
  \setcounter{philipenumida}{\value{enumi}}
  \addtocounter{philipenumida}{-1}
\end{enumerate}

\subsection{Discretization of Nonlocal Radiation Terms}\label{sec:nonlocdis}

\begin{enumerate}[({DA}-1)]
  \setcounter{enumi}{\thephilipenumida}
\item\label{eq:zetaalpha} 
  For a chosen fixed index ``ph'', 
  $(\zeta_{\alpha})_{\alpha\in I_\Omega}$
  and $(\zeta_{\alpha})_{\alpha\in I_\Sigma}$
  are finite partitions 
  of $\Gamma_\Omega$ and $\Sigma$, respectively, where
  \begin{equation}
    \label{eq:eqzetaalpha}
    I_\Omega\cap I_\Sigma=\emptyset,\q\q 
    \ph\notin I_\Omega\cup I_\Sigma,
  \end{equation}
  and, for each $\alpha\in I_\Omega$ (resp.\ $\alpha\in I_\Sigma$), 
  the boundary element
  $\zeta_\alpha$ is a nonvoid, polyhedral, connected, and 
  (relatively) open subset of $\Gamma_\Omega$ (resp.\ $\Sigma$),
  lying in a $2$-dimensional affine subspace of $\mathbb R^3$. 
  For the convenience of subsequent concise notation, let
  $\zeta_{\ph}:=\Gamma_{\ph}$
  and $I_\Gamma:=I_\Omega\dju\,\{\ph\}$. 
  \setcounter{philipenumida}{\value{enumi}}
\end{enumerate}
On both $\Gamma_{\Omega}$ and $\Sigma$, the boundary elements are supposed to be compatible
with the control volumes $\omega_i$:
\begin{enumerate}[({DA}-1)]
  \setcounter{enumi}{\thephilipenumida}
\item\label{eq:zetaomega} 
  For each $\alpha\in I_\Omega$ (resp.\ $\alpha\in I_\Sigma$), there is 
  a unique $i(\alpha)\in I$ such that 
  $\zeta_\alpha\subseteq\partial\omega_{i(\alpha)}\cap\Gamma_\Omega$
  (resp.\ $\zeta_\alpha\subseteq\partial\omega_{{\rm s},i(\alpha)}\cap\Gamma_\Sigma$).
  Moreover, for each $\alpha\in I_\Omega\dju I_\Sigma$:
  $x_{i(\alpha)}\in\overline\zeta_\alpha$. 
  \setcounter{philipenumida}{\value{enumi}}
\end{enumerate}
For each $i\in I$, define 
$J_{\Omega,i}:=\{\alpha\in I_\Omega:\,\lambda_2(\zeta_{\alpha}\cap\partial\omega_i)\not=0\}$
and $J_{\Sigma,i}:=\{\alpha\in I_\Sigma:\,\lambda_2(\zeta_{\alpha}\cap\partial\omega_{{\rm s},i})\not=0\}$.
\begin{remark}\label{rem:ii}
  As a consequence of \daref{item:spdis}, \daref{eq:zetaalpha}, and \daref{eq:zetaomega}, 
  the family $(\zeta_{\alpha}\cap\partial\omega_i)_{\alpha\in J_{\Omega,i}}$ 
  is a partition of $\partial\omega_i\cap\Gamma_{\Omega}=\partial\omega_{{\rm s},i}\cap\Gamma_{\Omega}$ and
  $(\zeta_{\alpha}\cap\partial\omega_{{\rm s},i})_{\alpha\in J_{\Sigma,i}}$ 
  is a partition of $\partial\omega_{{\rm s},i}\cap\Sigma=\overline\omega_i\cap\Sigma$.
  Moreover, \aref{eq:bndeqbnds} implies that at most
  one of the two sets $J_{\Omega,i}$, $J_{\Sigma,i}$ can be nonvoid
  (cf.\ Rem.\ \ref{rem:admomega} above). 
\end{remark}

Let
\begin{equation}
  \label{eq:lzetazetatildef}
  \Lambda_{\alpha,\beta}
  :=\int_{\zeta_{\alpha}\times\zeta_{\beta}}
  \Lambda\,\omega\q\q
  \text{for all $(\alpha,\beta)\in (I_\Sigma\times I_\Sigma)\cup(I_\Gamma\times I_\Gamma)$}.
\end{equation}
The $\Lambda_{\alpha,\beta}$ are nonnegative since $\Lambda\omega$ is nonnegative according to Th.\ \ref{th:kop}\eqref{item:thkopa}.
The forms of $\Lambda$ and $\omega$ imply the symmetry condition $\Lambda_{\alpha,\beta}=\Lambda_{\beta,\alpha}$;
and \eqref{eq:itemthkopa} (resp.\ its analog, where $\Sigma$ is replaced by $\Gamma=\Gamma_{\Omega}\cup\Gamma_{\rm ph}$) implies
\begin{equation}
  \label{eq:conserforlambdazeta}
  \sum_{\beta\in I_\Sigma}\Lambda_{\alpha,\beta}
  =\lambda_{2}({\zeta_{\alpha}})\q
  \text{for all $\alpha\in I_\Sigma$},\q\q
  \sum_{\beta\in I_\Gamma}\Lambda_{\alpha,\beta}
  =\lambda_{2}({\zeta_{\alpha}})\q
  \text{for all $\alpha\in I_\Gamma$}.
\end{equation}
Define vector-valued functions
\begin{subequations}
  \label{subeq:radmatform}
  {\allowdisplaybreaks
  \begin{align}
    \label{eq:radmatformy}
    \mbox{}
    \vec E_\Gamma&:\,(\mathbb R^+_0)^{I_\Omega}
    \longrightarrow(\mathbb R^+_0)^{I_\Omega},&
    \vec E_\Gamma(\vec u)
    =&\big(E_{\Gamma,\alpha}(\vec u)\big)_{\alpha\in I_\Omega},\nonumber\\*
    \vec E_\Sigma&:\,(\mathbb R^+_0)^{I_\Sigma}
    \longrightarrow(\mathbb R^+_0)^{I_\Sigma},&
    \vec E_\Sigma(\vec u)
    =&\big(E_{\Sigma,\alpha}(\vec u)\big)_{\alpha\in I_\Sigma},\nonumber\\*
    & &
    \hspace{-70pt}
    E_{\Gamma,\alpha}(\vec u)
    :=E_{\Sigma,\alpha}(\vec u)
    :=&\sigma\,\epsilon(u_\alpha)\, 
    u_\alpha^4\,\lambda_{2}({\zeta_{\alpha}}),\\
    \vec E_{\rm ph}&:\,(\mathbb R^+_0)^{I_\Omega}
    \longrightarrow(\mathbb R^+_0)^{I_\Omega},&
    \vec E_{\rm ph}(\vec u)
    =&\big(E_{\rm ph,\alpha}(\vec u)\big)_{\alpha\in I_\Omega},\nonumber\\*
    & &
    E_{\rm ph,\alpha}(\vec u)
    :=&\sigma\,\big(1-\epsilon(u_\alpha)\big)\,
    \theta_{\rm ext}^4\,\Lambda_{\alpha,\ph},
    \label{eq:radmatformyph}
  \end{align}}%
  and matrix-valued functions
  \begin{align}
    \vec G_\Gamma&:\,(\mathbb R^+_0)^{I_\Omega}\longrightarrow\mathbb R^{I_\Omega^2},&
    \vec G_\Gamma(\vec u)
    =&\big(G_{\Gamma,\alpha,\beta}(\vec u)\big)_{(\alpha,\beta)\in I_\Omega^2},\nonumber\\
    \vec G_\Sigma&:\,(\mathbb R^+_0)^{I_\Sigma}\longrightarrow\mathbb R^{I_\Sigma^2},&
    \vec G_\Sigma(\vec u)
    =&\big(G_{\Sigma,\alpha,\beta}(\vec u)\big)_{(\alpha,\beta)\in I_\Sigma^2},\nonumber\\
    & &
    \hspace{-80pt}
    G_{\Gamma,\alpha,\beta}(\vec u)
    :=G_{\Sigma,\alpha,\beta}(\vec u)
    :=&
    \begin{cases}
      \lambda_{2}({\zeta_{\alpha}})
      -\big(1-\epsilon(u_\alpha)\big)\,\Lambda_{\alpha,\beta}
      &\text{for $\alpha=\beta$},\\
      \hphantom{\lambda_{2}({\zeta_{\alpha}})}
      -\big(1-\epsilon(u_\alpha)\big)\,\Lambda_{\alpha,\beta}
      &\text{for $\alpha\not=\beta$}.
    \end{cases}
    \label{eq:newadef}
  \end{align}
\end{subequations}
\begin{lemma}
  \label{lem:mmatrix}
  The following holds for each $\vec u\in(\mathbb R^+_0)^{I_\Omega}$:
  \begin{enumerate}[{\bf(a)}]
  \item\label{item:mmatdiagdom} For each $\alpha\in I_\Omega$:
    $\sum_{\beta\in I_\Omega\setminus\{\alpha\}}
    |G_{\Gamma,\alpha,\beta}(\vec u)|\leq (1-\epsilon(u_\alpha))\,G_{\Gamma,\alpha,\alpha}(\vec u)
    < G_{\Gamma,\alpha,\alpha}(\vec u)$. In particular, $\vec G_\Gamma(\vec u)$ is strictly diagonally
    dominant. 
  \item\label{item:mmatmmat} $\vec G_\Gamma(\vec u)$ is an M-matrix, i.e.\ 
    $\vec G_\Gamma(\vec u)$ is invertible, $\vec G_\Gamma^{-1}(\vec u)$ is nonnegative, 
    and, for each $(\alpha,\beta)\in I_\Omega^2$ such that $\alpha\not=\beta$:
    $G_{\Gamma,\alpha,\beta}(\vec u)\leq0$.
  \end{enumerate}
  Analogous statements hold for $\vec G_\Sigma$.
\end{lemma}
\begin{proof}
  The proof is completely analogous to the proof of \cite[Lem.\ 3.1]{KP-05}.
\end{proof}

Using Lem.\ \ref{lem:mmatrix}, one can define the vector-valued functions
\begin{subequations}
  \label{subeq:rdef}
  \begin{align}
    \label{eq:xisainy}
    \vec R_\Gamma&:\,(\mathbb R^+_0)^{I_\Omega}\longrightarrow(\mathbb R^+_0)^{I_\Omega},&
    \vec R_\Gamma(\vec u)&:=\vec G_\Gamma^{-1}(\vec u)\,
    \big(\vec E_\Gamma(\vec u)+\vec E_{\rm ph}(\vec u)\big),\\
    \vec R_\Sigma&:\,(\mathbb R^+_0)^{I_\Sigma}\longrightarrow(\mathbb R^+_0)^{I_\Sigma},&
    \vec R_\Sigma(\vec u)
    &:=\vec G_\Sigma^{-1}(\vec u)\,\vec E_\Sigma(\vec u),\\
    \vec V_\Gamma&:\,(\mathbb R^+_0)^{I_\Omega}\longrightarrow(\mathbb R^+_0)^{I_\Omega},&
    \vec V_\Gamma(\vec u)&=\big(V_{\Gamma,\alpha}(\vec u)\big)_{\alpha\in I_\Omega},\nonumber\\
    \label{eq:nonlocsprintds}
    V_{\Gamma,\alpha}(\vec u)
    &:=\epsilon(u_\alpha)
    \sum_{\beta\in I_\Omega}R_{\Gamma,\beta}(\vec u)\,\Lambda_{\alpha,\beta}
    +\sigma\,\epsilon(u_\alpha)\,\theta_{\rm ext}^4\,\Lambda_{\alpha,\ph},\hspace{-90pt}& &\\[2mm]
    \vec V_\Sigma&:\,(\mathbb R^+_0)^{I_\Sigma}\longrightarrow(\mathbb R^+_0)^{I_\Sigma},&
    \vec V_\Sigma(\vec u)
    &=\big(V_{\Sigma,\alpha}(\vec u)\big)_{\alpha\in I_\Sigma},\nonumber\\
    V_{\Sigma,\alpha}(\vec u)&:=\epsilon(u_\alpha)\sum_{\beta\in I_\Sigma}
    R_{\Sigma,\beta}(\vec u)\,\Lambda_{\alpha,\beta}.& &
  \label{eq:disradtermsg}
  \end{align}
\end{subequations}

The following Lem.\ \ref{lem:nonnegrad} corresponds to 
\cite[Lem.\ 3.2]{KP-05} and provides a maximum principle as well as
local Lipschitzness for the functions $\vec R_\Gamma$, $\vec V_\Gamma$, $\vec R_\Sigma$, and $\vec V_\Sigma$.
Only part \eqref{item:nonnegradb} of Lem.\ \ref{lem:nonnegrad} is significantly different
from its counterpart in \cite[Lem.\ 3.2]{KP-05}. 
The following notation is introduced for
$\vec u=(u_i)_{i\in I}\in\mathbb R^I$ (where $I$ can be an arbitrary,
nonempty, finite index set):
\begin{equation}
  \label{eq:minandmax}
  \vecmin{\vec u}:=\min\{u_i:\,i\in I\},\q
  \vecmax{\vec u}:=\max\{u_i:\,i\in I\}.
\end{equation}
\begin{lemma}
  \label{lem:nonnegrad}
  \begin{enumerate}[{\rm\bf(a)}]
  \item\label{item:nonnegrada}
    The functions $\vec R_\Gamma$, $\vec V_\Gamma$, 
    $\vec R_\Sigma$, and $\vec V_\Sigma$ are all nonnegative.
  \item\label{item:nonnegradzero}
    For each $\vec u\in(\mathbb R^+_0)^{I_\Omega}$,
    $\alpha\in I_\Omega$: 
    \begin{subequations}     
      \begin{align}
        \sigma\,\min\big\{\vecmin{\vec u}^4,\theta_{\rm ext}^4\big\}
        \leq\mbox{}& R_{\Gamma,\alpha}(\vec u)
        \leq\sigma\,\max\big\{\vecmax{\vec u}^4,\theta_{\rm ext}^4\big\},\label{eq:nonnegradrevrg}\\ 
        \sigma\,\epsilon(u_\alpha)\,
        \min\big\{\vecmin{\vec u}^4,\theta_{\rm ext}^4\big\}\,
        \lambda_2(\zeta_\alpha) 
        \leq\mbox{}& V_{\Gamma,\alpha}(\vec u)\nonumber\\
        \leq\mbox{}&\sigma\,\epsilon(u_\alpha)\,\max\big\{\vecmax{\vec u}^4,\theta_{\rm ext}^4\big\}\,
        \lambda_2(\zeta_\alpha),\label{eq:nonnegradrevrgv}\\ 
        \intertext{and, for each $\vec u\in(\mathbb R^+_0)^{I_\Sigma}$,
          $\alpha\in I_\Sigma$:} 
        \sigma\,\vecmin{\vec u}^4
        \leq\mbox{}& R_{\Sigma,\alpha}(\vec u)
        \leq\sigma\,\vecmax{\vec u}^4,\label{eq:nonnegradrevrs}\\ 
        \sigma\,\epsilon(u_\alpha)\,\vecmin{\vec u}^4\,\lambda_2(\zeta_\alpha) \leq\mbox{}
        & 
        V_{\Sigma,\alpha}(\vec u)
        \leq\sigma\,\epsilon(u_\alpha)\,\vecmax{\vec u}^4\,
        \lambda_2(\zeta_\alpha).\label{eq:nonnegradrevrsv}
      \end{align}
    \end{subequations}
  \item\label{item:nonnegradb}
    For each $r\in\mathbb R^+$,
    with respect to the max-norm, 
    the maps $R_{\Gamma,\alpha}$, $V_{\Gamma,\alpha}$ are Lip\-schitz on $[0,r]^{I_\Omega}$,
    and the maps $R_{\Sigma,\alpha}$,
    and $V_{\Sigma,\alpha}$ are Lip\-schitz on $[0,r]^{I_{\Sigma}}$.
    More precisely,
    recalling $\epsilon_{\min,r}\in\mathbb R^+$ from Rem.\ {\rm\ref{rem:epsminr}} and
    $L_{\epsilon,r}$ from \aref{item:em}, the Lipschitz constants are
    \begin{subequations}
      \label{subeq:nonnegradb}
      \begin{align}
        &\sigma\,\epsilon_{\min,r}^{-1}\,\Big(4\,r^3
          +L_{\epsilon,r}\,\big(r^4+\max\{r^4,\theta_{\rm ext}^4\}\big)\Big)&
        &\text{for $\vec R_\Gamma$},\label{eq:nonnegradbrg}\\
        &\sigma\,\lambda_2(\zeta_\alpha)\,
        \Big(
        4\,\epsilon(u_\alpha)\,\epsilon_{\min,r}^{-1}\,r^3\nonumber\\
        &\hspace{50pt}+
        \max\{r^4,\theta_{\rm ext}^4\}\,
        L_{\epsilon,r}\,
        \big(2\,\epsilon(u_\alpha)\,\epsilon_{\min,r}^{-1}+1\big)
        \Big)&
        &\text{for $V_{\Gamma,\alpha}$},\label{eq:nonnegradbvg}\\
        &\sigma\,\epsilon_{\min,r}^{-1}\,\big(4\,r^3+2\,L_{\epsilon,r}\,r^4\big)\,&
        &\text{for $\vec R_\Sigma$},\label{eq:nonnegradbrs}\\
        &\sigma\,\lambda_2(\zeta_\alpha)\,
        \Big(
        4\,\epsilon(u_\alpha)\,\epsilon_{\min,r}^{-1}\,r^3
        +
        r^4\,L_{\epsilon,r}\,
        \big(2\,\epsilon(u_\alpha)\,\epsilon_{\min,r}^{-1}+1\big)
        \Big)&
        &\text{for $V_{\Sigma,\alpha}$}.\label{eq:nonnegradbvs}
      \end{align}
    \end{subequations}
  \end{enumerate}
\end{lemma}
\begin{proof}
  The proofs of \eqref{item:nonnegrada} and \eqref{item:nonnegradzero} are completely analogous to
  the proofs of \cite[Lem.\ 3.2(a),(b)]{KP-05}.

  \eqref{item:nonnegradb}: Note that \eqref{eq:xisainy} implies
  \begin{equation}
    \label{eq:radmatform}
    \vec G_\Gamma(\vec u)\,\vec R_\Gamma(\vec u)
    =\vec E_\Gamma(\vec u)
    +\vec E_{\rm ph}(\vec u),
  \end{equation}
  or, written in components:
  \begin{equation}
    \label{eq:radiositycylsystem}
    \begin{aligned}
      \mbox{}
      &R_{\Gamma,\alpha}(\vec u)\,\lambda_{2}({\zeta_{\alpha}})
      -\big(1-\epsilon(u_\alpha)\big)
      \sum_{\beta\in I_\Omega}R_{\Gamma,\beta}(\vec u)\,\Lambda_{\alpha,\beta}\\
      &=\sigma\,\epsilon(u_\alpha)\, 
      u_\alpha^4\,\lambda_{2}({\zeta_{\alpha}})
      +\sigma\,\big(1-\epsilon(u_\alpha)\big)\,
      \theta_{\rm ext}^4\,\Lambda_{\alpha,\ph}
    \end{aligned}\q\q\q\q
    (\alpha\in I_\Omega).
  \end{equation}
  The function $\theta\mapsto\theta^4$ is 
  $(4 r^3)$-Lipschitz on $[0,r]$, such that, 
  by \eqref{eq:radiositycylsystem},
  for each $(\vec u,\vec v)\in[0,r]^{I_\Omega}\times[0,r]^{I_\Omega}$,
  $\alpha\in I_\Omega$:
  \begin{align}
    &\left|\vphantom{\sum_{\beta\in I_\Omega}}\big(R_{\Gamma,\alpha}(\vec u)-R_{\Gamma,\alpha}(\vec v)\big)\,\lambda_{2}({\zeta_{\alpha}})\right.\nonumber\\
    &\q
      -
      \sum_{\beta\in I_\Omega}
      \Big(
      \big(1-\epsilon(u_{\alpha})\big)\,R_{\Gamma,\beta}(\vec u)
      -\big(1-\epsilon(v_{\alpha})\big)\,R_{\Gamma,\beta}(\vec v)
      \Big)
      \,\Lambda_{\alpha,\beta}\nonumber\\
    &\left.
      \vphantom{\sum_{\beta\in I_\Omega}}
      \q
      -\sigma\,
      \Big(
      \big(1-\epsilon(u_{\alpha})\big)
      -
      \big(1-\epsilon(v_{\alpha})\big)
      \Big)\,
      \theta_{\rm ext}^4\,\Lambda_{\alpha,\ph}\,
    \right|\nonumber\\[1mm]
    &=
    \sigma\, 
    \big|\epsilon(u_\alpha)\,u_{\alpha}^4-\epsilon(v_\alpha)\,v_{\alpha}^4\big|\,
    \lambda_{2}({\zeta_{\alpha}})
    \overset{\text{\aref{item:em}}}\leq
    \sigma\,\big(4\,r^3+L_{\epsilon,r}\,r^4\big)\,
    |u_{\alpha}-v_{\alpha}|\,\lambda_{2}({\zeta_{\alpha}}).
    \label{eq:olafrlip}
  \end{align}
  Now, let $\alpha\in I_\Omega$ be such that 
  $N_{\max}:=\|\vec R_\Gamma(\vec u)-\vec R_\Gamma(\vec v)\|_{\max}=|R_{\Gamma,\alpha}(\vec u)-R_{\Gamma,\alpha}(\vec v)|$.
  Then one can estimate
  \begin{align}
    \mbox{}
    &\bigg|
    \sigma\,
    \big(\epsilon(v_{\alpha})-\epsilon(u_{\alpha})\big)\,
    \theta_{\rm ext}^4\,\Lambda_{\alpha,\ph}\nonumber\\
    &\q+
    \sum_{\beta\in I_\Omega}
    \Big(
    \big(1-\epsilon(u_{\alpha})\big)\,R_{\Gamma,\beta}(\vec u)
    -\big(1-\epsilon(v_{\alpha})\big)\,R_{\Gamma,\beta}(\vec v)
    \Big)
    \,\Lambda_{\alpha,\beta}
    \bigg|\nonumber\\
    &\leq
    \sigma\,
    L_{\epsilon,r}\,\|\vec u-\vec v\|_{\max}\,
    \theta_{\rm ext}^4\,\Lambda_{\alpha,\ph}\nonumber\\
    &\q+
    \bigg|
    \sum_{\beta\in I_\Omega}
    \Big(
    \big(1-\epsilon(u_{\alpha})\big)\,R_{\Gamma,\beta}(\vec u)
    -\big(1-\epsilon(u_{\alpha})\big)\,R_{\Gamma,\beta}(\vec v)
    \Big)
    \,\Lambda_{\alpha,\beta}
    \bigg|\nonumber\\
    &\q+
    \bigg|
    \sum_{\beta\in I_\Omega}
    \Big(
    \big(1-\epsilon(u_{\alpha})\big)\,R_{\Gamma,\beta}(\vec v)
    -\big(1-\epsilon(v_{\alpha})\big)\,R_{\Gamma,\beta}(\vec v)
    \Big)
    \,\Lambda_{\alpha,\beta}
    \bigg|\nonumber\\[1mm]
    &\leq
    \sigma\,
    L_{\epsilon,r}\,\|\vec u-\vec v\|_{\max}\,
    \theta_{\rm ext}^4\,\Lambda_{\alpha,\ph}
    +
    \big(1-\epsilon(u_{\alpha})\big)\,N_{\max}\,
    \big(\lambda_{2}({\zeta_{\alpha}})-\Lambda_{\alpha,\ph}\big)\nonumber\\[1mm]
    &\q+
    \sigma\,\max\{r^4,\theta_{\rm ext}^4\}\,
    L_{\epsilon,r}\,\|\vec u-\vec v\|_{\max}\,
    \big(\lambda_{2}({\zeta_{\alpha}})-\Lambda_{\alpha,\ph}\big)\nonumber\\[1mm]
    &\leq
    \big(1-\epsilon(u_{\alpha})\big)\,N_{\max}\,\lambda_{2}({\zeta_{\alpha}})
    +
    \sigma\,\max\{r^4,\theta_{\rm ext}^4\}\,
    L_{\epsilon,r}\,\|\vec u-\vec v\|_{\max}\,\lambda_{2}({\zeta_{\alpha}}).
    \label{eq:olafrlipestnegterm}
  \end{align}
  Next, \eqref{eq:olafrlip} and \eqref{eq:olafrlipestnegterm} imply
  {\allowdisplaybreaks
  \begin{eqnarray*}
    & & \hspace{-50pt}
    \sigma\,\big(4\,r^3+L_{\epsilon,r}\,r^4\big)\,
    \|\vec u-\vec v\|_{\max}\,\lambda_{2}({\zeta_{\alpha}})\\*[1mm]
    &\overset{\eqref{eq:olafrlip}}\geq& 
    \left|
      N_{\max}\,\lambda_{2}({\zeta_{\alpha}})
      -
      \bigg|
      \sum_{\beta\in I_\Omega}
      \Big(
      \big(1-\epsilon(u_{\alpha})\big)\,R_{\Gamma,\beta}(\vec u)
      -\big(1-\epsilon(v_{\alpha})\big)\,R_{\Gamma,\beta}(\vec v)
      \Big)
      \,\Lambda_{\alpha,\beta}
    \right.\\*
    & &\left.
      \vphantom{\sum_{\beta\in I_\Omega}}
      \hspace{83pt}
      +\sigma\,
      \big(\epsilon(v_{\alpha})-\epsilon(u_{\alpha})\big)\,
      \theta_{\rm ext}^4\,\Lambda_{\alpha,\ph}\,
      \bigg|
    \right|\\
    &\geq& 
    N_{\max}\,\lambda_{2}({\zeta_{\alpha}})
    -
    \bigg|
    \sum_{\beta\in I_\Omega}
    \Big(
    \big(1-\epsilon(u_{\alpha})\big)\,R_{\Gamma,\beta}(\vec u)
    -\big(1-\epsilon(v_{\alpha})\big)\,R_{\Gamma,\beta}(\vec v)
    \Big)
    \,\Lambda_{\alpha,\beta}\\*
    & &
    \vphantom{\sum_{\beta\in I_\Omega}}
    \hspace{83pt}
    +\sigma\,
    \big(\epsilon(v_{\alpha})-\epsilon(u_{\alpha})\big)\,
    \theta_{\rm ext}^4\,\Lambda_{\alpha,\ph}\,
    \bigg|
    \\
    &\overset{\eqref{eq:olafrlipestnegterm}}\geq& N_{\max}\,
    \Big(\lambda_{2}({\zeta_{\alpha}})
    -\big(1-\epsilon(u_{\alpha})\big)\,\lambda_{2}({\zeta_{\alpha}})
    \Big)\\[1mm]
    & &
    -
    \sigma\,\max\{r^4,\theta_{\rm ext}^4\}\,
    L_{\epsilon,r}\,\|\vec u-\vec v\|_{\max}\,\lambda_{2}({\zeta_{\alpha}})
    \\[1mm]
    &=&
    N_{\max}\,\epsilon(u_{\alpha})\,\lambda_{2}({\zeta_{\alpha}})
    -
    \sigma\,\max\{r^4,\theta_{\rm ext}^4\}\,
    L_{\epsilon,r}\,\|\vec u-\vec v\|_{\max}\,\lambda_{2}({\zeta_{\alpha}}),
  \end{eqnarray*}}%
  thereby proving \eqref{eq:nonnegradbrg}.

  To prove the claimed Lipschitz continuity \eqref{eq:nonnegradbvg} of
  $\vec V_{\Gamma,\alpha}$, $\alpha\in I_\Omega$, one uses \eqref{eq:nonlocsprintds}
  and estimates
  \begin{eqnarray*}
    & & \hspace{-60pt}
    \big|V_{\Gamma,\alpha}(\vec u)-V_{\Gamma,\alpha}(\vec v)\big|\\
    &\overset{\text{\eqref{eq:nonlocsprintds},\aref{item:em}}}\leq&
    \left|
      \sum_{\beta\in I_\Omega}
      \big(
      \epsilon(u_\alpha)R_{\Gamma,\beta}(\vec u)
      -
      \epsilon(u_\alpha)R_{\Gamma,\beta}(\vec v)
      \big)\,\Lambda_{\alpha,\beta}
    \right|\\
    & &
    +\left|
      \sum_{\beta\in I_\Omega}
      \big(
      \epsilon(u_\alpha)R_{\Gamma,\beta}(\vec v)
      -
      \epsilon(v_\alpha)R_{\Gamma,\beta}(\vec v)
      \big)\,\Lambda_{\alpha,\beta}
    \right|\\[2mm]
    & &
    +\sigma\,L_{\epsilon,r}\,\big|u_\alpha-v_\alpha\big|\,\theta_{\rm ext}^4\,\Lambda_{\alpha,\ph}\\
    &\hspace{-6pt}\overset{\text{\eqref{eq:nonnegradbrg}, \aref{item:em}}}\leq&
    \epsilon(u_\alpha)\,
    \sigma\,\epsilon_{\min,r}^{-1}\,\Big(4\,r^3
    +L_{\epsilon,r}\,\big(r^4+\max\{r^4,\theta_{\rm ext}^4\}\big)\Big)\\
    & &
    \hspace{35pt}\cdot\|\vec u-\vec v\|_{\max}\,
    \big(\lambda_2(\zeta_\alpha)-\Lambda_{\alpha,\ph}\big)\\
    & &
    +\sigma\,\max\{r^4,\theta_{\rm ext}^4\}\,
    L_{\epsilon,r}\,\|\vec u-\vec v\|_{\max}\,
    \big(\lambda_2(\zeta_\alpha)-\Lambda_{\alpha,\ph}\big)\\
    & &
    +\sigma\,L_{\epsilon,r}\,\|\vec u-\vec v\|_{\max}\,\theta_{\rm ext}^4\,\Lambda_{\alpha,\ph}\\
    &\leq&
    \sigma\,
    \max\{r^4,\theta_{\rm ext}^4\}\,
    L_{\epsilon,r}\,
    \lambda_2(\zeta_\alpha)\,
    \big(2\,\epsilon(u_\alpha)\,\epsilon_{\min,r}^{-1}+1\big)\,
    \|\vec u-\vec v\|_{\max}\\
    & &
    +
    4\,\epsilon(u_\alpha)\,\sigma\,\epsilon_{\min,r}^{-1}\,r^3\,
    \lambda_2(\zeta_\alpha)\,
    \|\vec u-\vec v\|_{\max},
  \end{eqnarray*}
  which establishes \eqref{eq:nonnegradbvg}.

  The assertions \eqref{eq:nonnegradbrs} 
  and \eqref{eq:nonnegradbvs} 
  on $\vec R_\Sigma$ and $V_{\Sigma,\alpha}$, respectively, 
  can be proved analogously to the proofs of 
  \eqref{eq:nonnegradbrg} 
  and \eqref{eq:nonnegradbvg} above.
\end{proof}

\subsection{Formulation of Scheme}\label{sec:scheme}

Recalling the meaning of $i(\alpha)$ from \daref{eq:zetaomega},
for each $\vec u=(u_i)_{i\in I}$, define
\begin{equation}
  \label{eq:defures}
  \vec u\res{I_\Omega}:=(u_{i(\alpha)})_{\alpha\in I_\Omega},\q\q
  \vec u\res{I_\Sigma}:=(u_{i(\alpha)})_{\alpha\in I_\Sigma}.
\end{equation}
The finite volume scheme is now stated in \eqref{subeq:heatopeq} and 
\eqref{subeq:heatconsolfvol2} below. 
One is seeking a nonnegative solution 
$(\vec u_0,\dots,\vec u_N)$, $\vec u_\nu=(u_{\nu,i})_{i\in I}$, to 
\begin{subequations}
  \label{subeq:heatopeq}
  \begin{align}
    \label{eq:heatopeqin}
    u_{0,i}&=\theta_{{\rm init},i}& &(i\in I),\\
    \mathcal H_{\nu,i}(\vec u_{\nu-1},\vec u_\nu)&=0& &(i\in I,\q\nu\in\{1,\dots,N\}),
    \label{eq:heatopeq}
  \end{align}
\end{subequations}
where, for each $\nu\in\{1,\dots,N\}$:
{\allowdisplaybreaks
\begin{subequations}
  \label{subeq:heatconsolfvol2}
  \begin{align}
    \mbox{}
    \mathcal H_{\nu,i}&:\,(\mathbb R^+_0)^I\times(\mathbb R^+_0)^I\longrightarrow\mathbb R,\nonumber\\
    \mathcal H_{\nu,i}(\tilde{\vec u},\vec u)
    &=k_\nu^{-1}\,\sum_{m\in\{{\rm s,g}\}}
    \big(\ep_m(u_i)-\ep_m(\tilde u_i)\big)\,
    \lambda_3(\omega_{m,i})\label{eq:heatconsolfvoleps2}\\*
    &\q-\sum_{m\in\{{\rm s,g}\}}
    \kappa_m\,\sum_{j\in\nb_m(i)}
    \frac{u_j-u_i}{\|x_i-x_j\|_2}\,
    \lambda_{2}\big(\partial\omega_{m,i}\cap\partial\omega_{m,j}\big)
    \label{eq:heatconsolfvolkint2}\\*
    &\q+\sigma\,\epsilon(u_i)\,u_i^4\,
    \lambda_{2}\big(\partial\omega_{{\rm s},i}\cap\Gamma_{\Omega}\big)
    -\sum_{\alpha\in J_{\Omega,i}}V_{\Gamma,\alpha}(\vec u\res{I_\Omega})
    \label{eq:heatconsolfvolkgamma2}\\*
    &\q\vphantom{\sum_{\alpha\in I_i}}
    +\sigma\,\epsilon(u_i)\,
    (u_i^4-\,\theta_{\rm ext}^4)\,
    \lambda_{2}\big(\partial\omega_{{\rm s},i}\cap(\partial\Omega\setminus\Gamma_{\Omega})\big)
    \label{eq:heatconsolfvolknogamma2}\\*
    &\q+\sigma\,\epsilon(u_i)\,u_i^4\,
    \lambda_{2}\big(\omega_i\cap\Sigma\big)
    -\sum_{\alpha\in J_{\Sigma,i}}V_{\Sigma,\alpha}(\vec u\res{I_\Sigma})
    \label{eq:heatconsolfvolkif2}\\*
    &\q-\sum_{m\in\{{\rm s,g}\}}f_{m,\nu,i}\,
    \lambda_3(\omega_{m,i}),
    \label{eq:heatconsolfvolrhs2}
  \end{align}
\end{subequations}}%
where 
\begin{align}
  \label{eq:approxf}
  f_{m,\nu,i}
  \approx
  \frac{\int_{t_{\nu-1}}^{t_\nu}\int_{\omega_{m,i}}f_m}
  {k_\nu\,\lambda_3(\omega_{m,i})}
\end{align}
is a suitable approximation of the source term on $]t_{\nu-1},t_\nu[\times\omega_{m,i}$,
and $\theta_{{\rm init},i}$ is a suitable approximation of
$\theta_{\rm init}$ on $\omega_i$, $i\in I$. 
In general, the choices will depend
on the regularity of $f_m$ and $\theta_{\rm init}$ (for $f_m$ continuous, one might choose 
$f_{m,\nu,i}:=f_m(t_\nu,x_i)$, but 
$f_{m,\nu,i}:=(k_\nu\,\lambda_3(\omega_{m,i}))^{-1}\int_{t_{\nu-1}}^{t_\nu}\int_{\omega_{m,i}}f_m$
for a general $f_m\in L^\infty\big(]0,T[\times\Omega_m\big)$).
However, suitable approximations are assumed to satisfy:
\begin{enumerate}[({AA}-1)]
\item\label{item:fapprox} 
  For each $m\in\{{\rm s,g}\}$, $\nu\in\{0,\dots,N\}$, and $i\in I$:
  \begin{equation*}
    0
    \leq
    \essinf(f_m\res{]t_{\nu-1},\,t_\nu[\times\omega_{m,i}})
    \leq
    f_{m,\nu,i}
    \leq \|f_m\|_{L^\infty(]t_{\nu-1},\,t_\nu[\times\omega_{m,i})},
  \end{equation*}
  where $\essinf$ denotes the essential infimum.
\item\label{item:initapprox} 
  For each $i\in I$:
  \begin{equation*}
    0\leq\essinf(\theta_{\rm init}\res{\omega_i})\leq\theta_{{\rm init},i}
    \leq\|\theta_{\rm init}\|_{L^\infty(\omega_i)}.
  \end{equation*}
  \setcounter{philipenumiaa}{\value{enumi}}
\end{enumerate}

\begin{remark}
  As in \cite[Sec.\ 3.6]{KP-05}, one can consider the case where
  $\Omega_{\rm s}$ and $\Omega_{\rm g}$ are
  axisymmetric, and, in cylindrical coordinates $(r,\vartheta,z)$,
  the functions $\theta$,
  $f_{\rm s}$ and $f_{\rm g}$ are independent of the angular coordinate
  $\vartheta$, using the circular projection $(r,\vartheta,z)\mapsto (r,z)$ to
  reduce the model of Sec.\ \ref{sec:heat} as well as the finite
  volume scheme to two space dimensions. 
  As the arguments of \cite[Sec.\ 3.6]{KP-05} are still valid in the
  present fully implicit case,
  analogous
  reasoning to the contents of the following Sec.\ \ref{sec:sol} can
  be applied to the fully implicit axisymmetric finite volume scheme to prove
  a maximum principle as well as existence and uniqueness for the discrete
  solution, analogous
  to Th.\ \ref{th:sol}, Cor.\ \ref{cor:soln}, 
  and Cor.\ \ref{cor:solbigTnew} below.
\end{remark}

\section[Discrete Existence and Uniqueness]{Existence and Uniqueness of a Discrete Solution to the
  Finite Volume Scheme, Maximum Principle}\label{sec:sol}

As the proof of existence and uniqueness of a discrete
solution to the finite volume scheme in \cite{KP-05}, the proof of
existence and uniqueness of a discrete solution to the fully implicit
finite volume scheme \eqref{subeq:heatopeq} in Th.\ \ref{th:sol}
and Cor.\ \ref{cor:solbigTnew} below is based on the root problem 
with maximum principle \cite[Th.\ 4.1]{KP-05}. For the convenience of
the reader, \cite[Th.\ 4.1]{KP-05} is now reproduced as Th.\ \ref{th:opzeropar}:
\begin{theorem}
  \label{th:opzeropar}
  Let $\tau\subseteq\mathbb R$ be a (closed, open, half-open, bounded
  or unbounded) interval.
  Given a finite, nonempty index set $I$, 
  consider a continuous operator 
  \begin{equation}
    \label{eq:rigop}
    \mathcal H:\,\tau^I\longrightarrow\mathbb R^I,\q 
    \mathcal H(\vec u)=\big(\mathcal H_i(\vec u)\big)_{i\in I}.
  \end{equation}
  Assume there are continuous functions 
  $b_i\in C(\tau,\mathbb R)$, 
  $\tilde h_i\in C(\tau,\mathbb R)$, 
  $\tilde g_i\in C(\tau^I,\mathbb R)$, $i\in I$, 
  such that the following
  conditions {\rm \eqref{item:hhgpar} -- \eqref{item:biinvlip}} are satisfied. 
  \begin{enumerate}[\em(i)]
  \item\label{item:hhgpar} 
    There is $\tilde{\vec u}\in\tau^I$ such that,
    for each $i\in I$, $\vec u\in\tau^I$:
    \begin{equation*}
      \mathcal H_i(\vec u)
      =b_i(u_i)+\tilde h_i(u_i)-b_i({\tilde u}_i)-\tilde g_i(\vec u).
    \end{equation*}
  \item\label{item:bfmtilde} There are $\tilde m,\tilde M\in\tau$,
    a family of nonpositive numbers
    $(\beta_i)_{i\in I}\in(\mathbb R^-_0)^I$, 
    and a family of nonnegative numbers
    $(B_i)_{i\in I}\in(\mathbb R^+_0)^I$ such that, for each $i\in I$, 
    $\vec u\in\tau^I$, $\theta\in\tau$:
    \begin{subequations}
      \label{subeq:bfmtilde}
      \begin{align}
      \label{eq:bfmtildemax}
        &\max\big\{\vecmax{\vec u},\,\tilde M\big\}\leq\theta& 
        &\Rightarrow&
        &\tilde g_i(\vec u)-\tilde h_i(\theta)\leq B_i,\\
        &\theta\leq\min\big\{\tilde m,\,\vecmin{\vec u}\big\}& 
        &\Rightarrow&
        &\tilde g_i(\vec u)-\tilde h_i(\theta)\geq \beta_i,
      \label{eq:bfmtildemin}
      \end{align}
    \end{subequations}
    where $\vecmax{\vec u}$ and $\vecmin{\vec u}$ are according to
    \eqref{eq:minandmax}. 
  \item\label{item:biinvlip} 
    There is a family of positive numbers 
    $(C_{b,i})_{i\in I}\in(\mathbb R^+)^I$ such that,
    for each $i\in I$ and $\theta_1,\theta_2\in\tau$:\q\;
    $\theta_2\geq\theta_1\q\Rightarrow\q
    b_i(\theta_2)\geq (\theta_2-\theta_1)\,C_{b,i}+b_i(\theta_1)$.
  \setcounter{philipenumi}{\value{enumi}}
  \end{enumerate}
  Letting
  \begin{align}
    \label{eq:maxbidef}
    \beta&:=\min\left\{\frac{\beta_i}{C_{b,i}}:\,i\in I\right\},&
    B&:=\max\left\{\frac{B_i}{C_{b,i}}:\,i\in I\right\},\\[1mm]
    m(\tilde{\vec u})&:=\min\big\{\tilde m,\,\vecmin{\tilde{\vec u}}+\beta\big\},&
    M(\tilde{\vec u})&:=\max\big\{\tilde M,\,\vecmax{\tilde{\vec u}}+B\big\},
    \label{eq:mtiludef}
  \end{align}
  one has the following maximum principle: If $\vec u_0\in\tau^I$
  satisfies $\mathcal H(\vec u_0)=\vec 0:=(0,\dots,0)$,
  then $\vec u_0\in[m(\tilde{\vec u}),M(\tilde{\vec u})]^I$.

  If, in addition to {\rm \eqref{item:hhgpar} -- \eqref{item:biinvlip}},
  the following conditions {\rm \eqref{item:tildegilip} -- \eqref{item:lglesslhpar}}
  are satisfied, then there is a unique $\vec u_0\in[m(\tilde{\vec u}),M(\tilde{\vec u})]^I$ such that
  $\mathcal H(\vec u_0)=\vec 0$. 
  \begin{enumerate}[\em(i)]
  \setcounter{enumi}{\thephilipenumi}
  \item\label{item:tildegilip}
    For each $i\in I$, there is 
    $L_{g,i}(\tilde{\vec u})\in\mathbb R^+_0$ such that 
    $\tilde g_i$ is $L_{g,i}(\tilde{\vec u})$-Lipschitz with respect to the
    max-norm on $[m(\tilde{\vec u}),M(\tilde{\vec u})]^I$.
  \item\label{item:tildehinvlip}
    For each $i\in I$, there is 
    $C_{\tilde h,i}(\tilde{\vec u})\in\mathbb R^+_0$ such that, 
    for each $\theta_1,\theta_2\in[m(\tilde{\vec u}),M(\tilde{\vec u})]$:\q\;
    $\theta_2\geq\theta_1\q\Rightarrow\q
    \tilde h_i(\theta_2)\geq (\theta_2-\theta_1)\,C_{\tilde h,i}(\tilde{\vec u})
    +\tilde h_i(\theta_1)$.
  \item\label{item:lglesslhpar} 
    $L_{g,i}(\tilde{\vec u})<C_{b,i}+C_{\tilde h,i}(\tilde{\vec u})$ 
    for each $i\in I$.
  \end{enumerate}
\end{theorem}
\begin{proof}
  See \cite[Th.\ 4.1]{KP-05}.
\end{proof}

As in \cite{KP-05}, the essential step in proving the
discrete existence and uniqueness results is to first provide a discrete 
existence result with maximum principle, locally in time.
This is accomplished by the following Th.\ \ref{th:sol} which
corresponds to \cite[Th.\ 4.2]{KP-05}. Given an 
arbitrary vector $\tilde{\vec u}\in(\mathbb R^+_0)^I$,
Th.\ \ref{th:sol} establishes that each root of the finite volume scheme operator
$\mathcal H_\nu(\tilde{\vec u},\cdot)$ of \eqref{subeq:heatconsolfvol2}
satisfies a maximum principle. Moreover, 
Th.\ \ref{th:sol} proves the existence of a unique root to
$\mathcal H_\nu(\tilde{\vec u},\cdot)$,
provided that the $\nu$-th time step $k_\nu$ is sufficiently small. 

As in \cite[Th.\ 4.2]{KP-05}, 
the upper and lower bound for the solution, respectively, given by \eqref{eq:def-mnu} and 
\eqref{eq:def-Mnu} below, are determined by the external temperature
$\theta_{\rm ext}$, by the max and min of $\tilde{\vec u}$
as defined in \eqref{eq:minandmax}, 
by the size of the time step, and by the values of 
the heat sources in the time interval $[t_{\nu-1},t_\nu]$. 

The condition on the time step size \eqref{eq:cond-knu} arises from
the radiation terms in \eqref{subeq:heatconsolfvol2}, namely,
\eqref{eq:heatconsolfvolkgamma2} -- \eqref{eq:heatconsolfvolkif2}. 
It depends on the constants
$L_{\vec V,\Gamma_{\Omega}}$, 
$L_{\vec V,\partial\Omega\setminus\Gamma_{\Omega}}$, 
and $L_{\vec V,\Sigma}$ 
defined in \eqref{eq:def-LV} -- \eqref{eq:def-LVsigma} below, involving the ratios
between the size of boundary elements and adjacent volume
elements. Thus, these constants are of order $h^{-1}$ if $h$ is a
parameter for the fineness of a space discretization constructed by
uniform refinement of some initial grid, such that 
\eqref{eq:cond-knu} is of the form $k\sim h$ in the notation of 
the Introduction.

Letting $\tilde{\vec u}=\vec u_{\nu-1}$, 
as a direct consequence of Th.\ \ref{th:sol}, for $k_\nu$ small enough,
each nonnegative solution $(\vec u_0, \dots, \vec u_{\nu-1})$ 
to the finite volume scheme \eqref{subeq:heatopeq}
with $N$ replaced by $\nu-1 < N$, can be uniquely extended to $t=t_\nu$
(see Cor.\ \ref{cor:soln}). 

Finally, as in \cite[Th.\ 4.3]{KP-05}, an inductive argument extends
the local result of Th.\ \ref{th:sol} to guarantee a unique solution to the entire
finite volume scheme \eqref{subeq:heatopeq} (see Cor.\ \ref{cor:solbigTnew}).

In preparation for Th.\ \ref{th:sol}, notions of the {\em variation} 
of a function are recalled as well as some elementary properties: For a function
$f:\,[a,\infty[\longrightarrow\mathbb R$, let 
\begin{equation}
  \label{eq:variation}
  \begin{gathered}
    \var^+ f:\,[a,\infty[\longrightarrow[0,\infty],\\
    \var^+ f(a):=0,\\
      \begin{aligned}
        \var^+ f(\lambda):=\sup&\left\{
          \sum_{\nu=1}^N\max\big\{0,f(t_\nu)-f(t_{\nu-1})\big\}:\right.\\
        &\left.\hspace{15pt}
          \vphantom{\underset{\nu\in\{1,\dots,N\}}\sum}
          (t_\nu)_{\nu\in\{0,\dots,N\}}\text{ is a discretization of $[a,\lambda]$}
        \right\}\;
        \text{for $\lambda>a$},
      \end{aligned}
  \end{gathered}
\end{equation}
denote the {\em positive variation} of $f$, and define its 
{\em negative variation} $\var^- f$ by replacing ``$\max$'' in 
\eqref{eq:variation} with ``$-\min$''. 
Then, $\var^+ f$ and $\var^- f$ are nonnegative and
increasing. Moreover, if $f$ is $L$-Lipschitz on $[a,r]$, then
$\var^+ f$ and $\var^- f$ are $L$-Lipschitz on $[a,r]$, and,
for each $\lambda\in[a,r]$, 
$f(\lambda)=f(a) + \var^+ f(\lambda) - \var^- f(\lambda)$.

\begin{theorem}
  \label{th:sol}
  Assume \emaref{item:aomega} -- \emaref{item:initial}, 
  \emdaref{item:spdis} -- \emdaref{eq:zetaomega}, 
  \emaaref{item:fapprox} and \emaaref{item:initapprox}.
  Moreover, assume $\nu\in\{1,\dots,N\}$ and 
  $\tilde{\vec u} =({\tilde u}_i)_{i\in I}
  \in  \left( \mathbb R^+_0\right)^I $.
  Let 
  \begin{subequations}
    \label{subeq:thsolconstdef}
    \begin{align}
      B_{f,\nu }  & :=  \max \left\{
        \sum_{m\in\{{\rm s,g}\}} f_{m,\nu,i}\, 
        \frac{\lambda_3(\omega_{m,i})}{\lambda_3(\omega_{i})} : \,
        i \in I \right\},
      \label{eq:def-Bf-nu}
      \displaybreak[2]
      \\ 
      L_{\vec V,\Gamma_{\Omega}}
      &:=\sigma\,\max\left\{
        \frac{ \lambda_2(\partial\omega_{{\rm s},i}\cap\Gamma_{\Omega})}{\lambda_3(\omega_{i})}:\,i\in I\right\},
      \label{eq:def-LV}
      \displaybreak[2]
      \\ 
      L_{\vec V,\partial\Omega\setminus\Gamma_{\Omega}}
      &:=\sigma\,\max\left\{
        \frac{ \lambda_2(\partial\omega_{{\rm s},i}
      \cap(\partial\Omega\setminus\Gamma_{\Omega}))}{\lambda_3(\omega_{i})}:\,i\in I\right\},
      \label{eq:def-LVomega}
      \displaybreak[2]
      \\ 
      L_{\vec V,\Sigma}
      &:=\sigma\,\max\left\{
        \frac{ \lambda_2(\omega_i \cap \Sigma)}{\lambda_3(\omega_{i})}:\,i\in I\right\},
      \label{eq:def-LVsigma}
      \displaybreak[2]
      \\[1mm]
      m(\tilde{\vec u}) &:= \min\big\{\,\theta_{\rm ext},\,  
      \vecmin{\tilde{\vec u}} \big\} ,
      \label{eq:def-mnu}
      \\[1mm]
      M_\nu(\tilde{\vec u}) &:= \max \left\{  \theta_{\rm ext},\, 
        \vecmax{\tilde{\vec u}} + \frac{k_\nu}{C_\ep} B_{f,\nu}
      \right\},
      \label{eq:def-Mnu}
    \end{align}
  \end{subequations}
  with $\vecmin{\tilde{\vec u}}$, 
  $\vecmax{\tilde{\vec u}}$ according to \eqref{eq:minandmax}, 
  and $C_{\ep}$ according to \emaref{item:ae}.
  
  Then, one has the maximum principle that each solution
  $\vec u_{\nu} =(u_{\nu,i})_{i\in I} \in 
  \left( \mathbb R^+_0\right)^I $ to
  \begin{equation}
    \mathcal H_{\nu,i}(\tilde{\vec u},\vec u_\nu)=0 \q\q (i\in I)
    \label{eq:heatopeqnu}
  \end{equation}
  must lie in 
  $[m(\tilde{\vec u}),M_\nu(\tilde{\vec u})]^I$. Furthermore, if
  \begin{subequations}
    \label{subeq:lgdef}
    {\allowdisplaybreaks
    \begin{align}
      \label{eq:def-lgone}
      l_g^1 &:\,\mathbb R^+_0\longrightarrow\mathbb R^+_0,\q
      l_g^1(r) := 4\,\var^-\epsilon(r)\,r^3 + L_{\epsilon,r}\,r^4,
      \\[2mm]
      \label{eq:def-lgtwo}
      l_g^2 &:\,\mathbb R^+_0\longrightarrow\mathbb R^+_0,\q
      l_g^2(r) := 4\,\epsilon_{\min,r}^{-1}\,r^3 + \max\{r^4,\theta_{\rm ext}^4\}\,L_{\epsilon,r}\,\big(2\,\epsilon_{\min,r}^{-1}+1\big),
      \\[2mm]
      \label{eq:def-lgthree}
      l_g^3 &:\,\mathbb R^+_0\longrightarrow\mathbb R^+_0,\q
      l_g^3(r) := L_{\epsilon,r}\,\theta_{\rm ext}^4,    
    \end{align}}%
  \end{subequations}
  where $L_{\epsilon,r}$ and $\epsilon_{\min,r}$ are according to
  \aref{item:em} and Rem.\ {\rm\ref{rem:epsminr}}, respectively,
  and, if $k_\nu$ is such that 
  \begin{align}
    \mbox{}&
    k_\nu\,\Big((L_{\vec V,\Gamma_{\Omega}}+L_{\vec V,\partial\Omega\setminus\Gamma_{\Omega}}+L_{\vec V,\Sigma})
    \,l_g^1\big(M_\nu(\tilde{\vec u})\big)
    +
    (L_{\vec V,\Gamma_{\Omega}} + L_{\vec V,\Sigma})
    \,l_g^2\big(M_\nu(\tilde{\vec u})\big)
    \nonumber\\
    &\q\;
    +L_{\vec V,\partial\Omega\setminus\Gamma_{\Omega}}\,
    l_g^3\big(M_\nu(\tilde{\vec u})\big)\Big)
    \;<\; C_\ep,
    \label{eq:cond-knu}
  \end{align}
  then there is a unique 
  $\vec u_{\nu}\in [m(\tilde{\vec u}),M_\nu(\tilde{\vec u})]^I$
  satisfying \eqref{eq:heatopeqnu}. 
\end{theorem}
\begin{proof}
  First note that, by choosing $k_\nu$ sufficiently small, one can ensure 
  that \eqref{eq:cond-knu} is satisfied: Since all three functions 
  $l_g^1$, $l_g^2$, $l_g^3$ are increasing, it follows from 
  \eqref{eq:def-Mnu} that, by decreasing $k_\nu$, one decreases both factors 
  on the left-hand side of \eqref{eq:cond-knu}.

  Now, the goal is to apply Th.\ \ref{th:opzeropar} with 
  $\tau=\mathbb R^+_0$ and 
  $\mathcal H_\nu(\tilde{\vec u},\cdot)$ playing the role of $\mathcal H$. 
  To that end, in the following, one defines continuous functions 
  $\myb$, $\tildeh$, $\tildeg$, 
  as well as numbers $\tilde m,\tilde M\in\mathbb R^+_0$, 
  $\beta_i\in\mathbb R^-_0$, $B_{\nu,i}\in\mathbb R^+_0$, 
  $C_{b,\nu,i}\in\mathbb R^+$, $L_{g,\nu,i}(\tilde{\vec u})\in\mathbb R^+$, 
  and $C_{\tilde h,\nu,i}(\tilde{\vec u})\in\mathbb R^+$
  that satisfy the hypotheses of Th.\ \ref{th:opzeropar} 
  (where the
  quantities with index $\nu$ correspond to the matching quantities
  without index $\nu$ in Th.\ \ref{th:opzeropar}). 
  Condition \eqref{eq:cond-knu} will {\em only} be needed to prove
  hypothesis \eqref{item:lglesslhpar} of Th.\ \ref{th:opzeropar}. 

  For each $i\in I$, let
  \begin{subequations}
    \label{subeq:funcsforlem}
    {\allowdisplaybreaks
    \begin{align}
      \label{eq:deftildebnu}
      \myb&:\,\mathbb R^+_0\longrightarrow\mathbb R^+_0,\q\q\q
      \myb(\theta):=k_\nu^{-1}\sum_{m\in\{{\rm s,g}\}}
      \ep_m(\theta)\,\lambda_3(\omega_{m,i}),\\[2mm]
      \label{eq:def-Lkappai}
      L_{\kappa,i} & := \sum_{m\in\{{\rm s,g}\}}
      \kappa_m\,\sum_{j\in\nb_m(i)}
      \frac{\lambda_{2}
        \big(\partial \omega_{m,i}\cap\partial\omega_{m,j}\big)}{\|x_i-x_j\|_2}
      \ge 0,
      \\[2mm]
      C_{\vec V,i}  &:= 
      \sigma\,\lambda_{2}\big(\partial\omega_{{\rm s},i}\cap\Gamma_{\Omega}\big)
      + \sigma\,\lambda_{2}\big(\partial\omega_{{\rm s},i}
      \cap(\partial\Omega\setminus\Gamma_{\Omega})\big)
      + \sigma\,\lambda_{2}(\omega_i\cap\Sigma)  \ge 0, 
      \label{eq:def-CVnui}
      \\[2mm]
      \tildeh&:\,\mathbb R^+_0\longrightarrow\mathbb R,
      \nonumber\\*
      \tildeh(\theta)&:=
      L_{\kappa,i}\,\theta
      +
      C_{\vec V,i}\,
      \big(
      \epsilon(0)
      +\var^+\epsilon(\theta)
      \big)\,\theta^4\, 
      \nonumber\\*
      &\q
      +\sigma\,\big(\var^-\epsilon(\theta)-\epsilon(0)\big)\,\theta_{\rm ext}^4\,
      \lambda_{2}\big(\partial\omega_{{\rm s},i}\cap(\partial\Omega\setminus\Gamma_{\Omega})\big),
      \label{eq:deftildehnu}
      \\[2mm] 
      \tildeg&:\,(\mathbb R^+_0)^I\longrightarrow\mathbb R^+_0,\nonumber\\*
      \tildeg(\vec u)&:=\sum_{m\in\{{\rm s,g}\}}
      \kappa_m\,\sum_{j\in\nb_m(i)}
      \frac{u_j}{\|x_i-x_j\|_2}\,
      \lambda_{2}(\partial \omega_{m,i}\cap\partial \omega_{m,j})
      \nonumber\\*[2mm]
      &\q+C_{\vec V,i}\,\var^-\epsilon(u_i)\,u_i^4
      \nonumber\\*[2mm]
      &\q
      +\sum_{\alpha\in J_{\Omega,i}}V_{\Gamma,\alpha}(\vec u\res{I_\Omega})
      +\sigma\,\var^+\epsilon(u_i)\,\theta_{\rm ext}^4\,
      \lambda_{2}\big(\partial\omega_{{\rm s},i}\cap(\partial\Omega\setminus\Gamma_{\Omega})\big)
      \nonumber\\*[0.5mm]
      &\q+\sum_{\alpha\in J_{\Sigma,i}}V_{\Sigma,\alpha}(\vec u\res{I_\Sigma})
      +\sum_{m\in\{{\rm s,g}\}}f_{m,\nu,i}\,\lambda_3(\omega_{m,i}),
      \label{eq:deftildegnu}\\[2mm]
      \tilde m&:=\tilde M:=\theta_{\rm ext},\q\q
      \beta_i:=0,\q\q
      B_{\nu,i}:=\sum_{m\in\{{\rm s,g}\}}f_{m,\nu,i}\,\lambda_3(\omega_{m,i}),
      \label{eq:defmM}\\[2mm]
      C_{b,\nu,i}&:=k_\nu^{-1}\,C_\ep\,\lambda_3(\omega_i) > 0,
      \label{eq:defcbnui}\\[2mm]
      \label{eq:def-Lgnui}
      L_{g,\nu,i}(\tilde{\vec u}) &:=  
      L_{\kappa,i}
      +C_{\vec V,i}\,\,l_g^1\big(M_\nu(\tilde{\vec u})\big)
      +
      \sigma\,
      \big(
      \lambda_2(\partial\omega_{{\rm s},i}\cap\Gamma_{\Omega})
      +\lambda_2(\omega_i\cap\Sigma)
      \big)
      \,l_g^2\big(M_\nu(\tilde{\vec u})\big)
      \nonumber\\*
      &\q
      +\sigma\,\lambda_{2}\big(\partial\omega_{{\rm s},i}\cap(\partial\Omega\setminus\Gamma_{\Omega})\big)\,
      l_g^3\big(M_\nu(\tilde{\vec u})\big)
      \ge 0,
      \\[2mm]
      \label{eq:def-Chnui}
      C_{h,\nu,i}(\tilde{\vec u}) &:= C_{b,\nu,i} + 
      L_{\kappa,i} + 4\, m(\tilde{\vec u})^3\,\epsilon(0)\,C_{\vec V,i}> 0. 
    \end{align}}%
  \end{subequations}
  Note that the numbers $m(\tilde{\vec u})$ and $M_\nu(\tilde{\vec u})$
  defined in \eqref{eq:def-mnu} and \eqref{eq:def-Mnu}, respectively, 
  correspond to the numbers $m(\tilde{\vec u})$ and 
  $M(\tilde{\vec u})$ as defined in \eqref{eq:mtiludef} in Th.\ \ref{th:opzeropar}. 

  It remains to verify the hypotheses 
  \eqref{item:hhgpar} -- \eqref{item:lglesslhpar} of Th.\ \ref{th:opzeropar}. 

  Th.\ \ref{th:opzeropar}\eqref{item:hhgpar}: 
  Showing $\mch(\tilde{\vec u},\vec u)
  =\myb(u_i)+\tildeh(u_i)-\myb({\tilde u}_i)-\tildeg(\vec u)$ 
  is straightforward from the respective definitions in \eqref{subeq:funcsforlem}.

  Th.\ \ref{th:opzeropar}\eqref{item:bfmtilde}: 
  One has to show that, for each $i\in I$, 
  $\vec u\in(\mathbb R^+_0)^I$,
  $\theta\in\mathbb R^+_0$:
  \begin{subequations}
    \label{subeq:bfmtildepr}
    \begin{align}
      \label{eq:bfmtildemaxpr}
      &\max\big\{\vecmax{\vec u},\,\theta_{\rm ext}\big\}\leq\theta& 
      &\Rightarrow&
      &\tildeg(\vec u)-\tildeh(\theta)
      \leq B_{\nu,i},\\
      &\theta\leq\min\big\{\theta_{\rm ext},\,\vecmin{\vec u}\big\}& 
      &\Rightarrow&
      &\tildeg(\vec u)-\tildeh(\theta)\geq 0.
      \label{eq:bfmtildeminpr}
    \end{align}
  \end{subequations}
  Considering 
  Lem.\ \ref{lem:nonnegrad}\eqref{item:nonnegradzero} 
  and Rem.\ \ref{rem:ii}, one sees that 
  \begin{align*}
    \sum_{\alpha\in J_{\Omega,i}}V_{\Gamma,\alpha}(\vec u\res{I_\Omega})
    &\leq
    \sigma\,\epsilon(u_i)
    \max\big\{\vecmax{\vec u}^4,\,\theta_{\rm ext}^4\big\}\,
    \lambda_2(\partial \omega_{s,i} \cap \Gamma_\Omega ),
    \\ 
    \sum_{\alpha\in J_{\Sigma,i}}
    V_{\Sigma,\alpha}(\vec u\res{I_\Sigma})
    &\leq
    \sigma\,\epsilon(u_i)\, \vecmax{\vec u}^4 \,
    \lambda_2(\omega_i \cap \Sigma).
  \end{align*}
  If $\theta\geq\theta_{\rm ext}$ and $\theta\geq\vecmax{\vec u}$, then,
  noting that $\epsilon(\theta)+\var^-\epsilon(\theta)=\epsilon(0)+\var^+\epsilon(\theta)$,
  and, by recalling 
  \eqref{eq:def-Bf-nu} 
  and \eqref{eq:def-Lkappai} -- \eqref{eq:defmM}, one obtains 
  \begin{align*}
    \tildeg(\vec u) \le  \mbox{} &
    \sum_{m\in\{{\rm s,g}\}}
    \kappa_m\,\sum_{j\in\nb_m(i)}
    \frac{\theta}{\|x_i-x_j\|_2}\,
    \lambda_{2}\big(\partial \omega_{m,i}\cap\partial \omega_{m,j}\big)
    \nonumber\\[1mm]
    &
    +C_{\vec V,i}\,\var^-\epsilon(\theta)\,\theta^4
    \nonumber\\
    &
    +\sigma\,\epsilon(\theta)\,\theta^4\,
    \lambda_2(\partial \omega_{s,i} \cap \Gamma_\Omega)
    \nonumber\\
    &
    +\sigma\,\epsilon(\theta)\,(\theta^4-\theta_{\rm ext}^4)\,
    \lambda_{2}\big(\partial\omega_{{\rm s},i}\cap(\partial\Omega\setminus\Gamma_{\Omega})\big)
    \nonumber\\
    &
    +\sigma\,\var^+\epsilon(\theta)\,\theta_{\rm ext}^4\,
    \lambda_{2}\big(\partial\omega_{{\rm s},i}\cap(\partial\Omega\setminus\Gamma_{\Omega})\big)
    \nonumber\\[0.5mm]
    &
    +\sigma\,\epsilon(\theta)\, \theta^4 \,
    \lambda_2(\omega_i \cap \Sigma)
    +\sum_{m\in\{{\rm s,g}\}}f_{m,\nu,i}\,\lambda_3(\omega_{m,i})
    \nonumber \\
    = \mbox{} & 
    \theta\,L_{\kappa,i} 
    +
    C_{\vec V,i}\,
    \big(
    \epsilon(0)
    +\var^+\epsilon(\theta)
    \big)\,\theta^4
    \nonumber\\
    &
    +\sigma\,\big(\var^-\epsilon(\theta)-\epsilon(0)\big)\,\theta_{\rm ext}^4\,
    \lambda_{2}\big(\partial\omega_{{\rm s},i}\cap(\partial\Omega\setminus\Gamma_{\Omega})\big)
    + \sum_{m\in\{{\rm s,g}\}}f_{m,\nu,i}\,\lambda_3(\omega_{m,i})\nonumber\\
    = \mbox{} &
    \tildeh(\theta) + B_{\nu,i},
  \end{align*}
  proving \eqref{eq:bfmtildemaxpr}. On the other hand, if 
  $\theta\leq\theta_{\rm ext}$ and $\theta\leq\vecmin{\vec u}$, 
  then, as $f_{m,\nu,i} \ge 0$ by \aaref{item:fapprox}, an analogous computation shows 
  $\tildeg(\vec u) \ge \tildeh(\theta)$, proving \eqref{eq:bfmtildeminpr}. 

  Th.\ \ref{th:opzeropar}\eqref{item:biinvlip}: That, 
  for $\theta_2\geq\theta_1\geq 0$, one has
  $\myb(\theta_2)\geq(\theta_2-\theta_1)\,C_{b,\nu,i}+\myb(\theta_1)$
  is immediate from combining \aref{item:ae}, 
  \eqref{eq:deftildebnu}, and \eqref{eq:defcbnui}. 

  Th.\ \ref{th:opzeropar}\eqref{item:tildegilip}: 
  For each $i \in I$, 
  one has to show that $\tildeg$ 
  is $L_{g,\nu,i}(\tilde{\vec u})$-Lipschitz with respect 
  to the max-norm on $[m(\tilde{\vec u}),M_\nu(\tilde{\vec u})]^I$. 
  The function
  \begin{equation*}
    \vec u\mapsto\sum_{m\in\{{\rm s,g}\}}
    \kappa_m\,\sum_{j\in\nb_m(i)}
    \frac{u_j}{\|x_i-x_j\|_2}\,
    \lambda_{2}(\partial\omega_{m,i}\cap\partial\omega_{m,j})\q\q
    \big(\vec u\in(\mathbb R^+_0)^I\big)
  \end{equation*}
  is $L_{\kappa,i}$-Lipschitz, $L_{\kappa,i}$ 
  according to \eqref{eq:def-Lkappai}, and,
  using \eqref{eq:def-lgone}, the map
  \begin{equation*}
    \vec u
    \mapsto
    C_{\vec V,i}\,\var^-\epsilon(u_i)\,u_i^4
    \q\q
    \big(\vec u\in[0,M_\nu(\tilde{\vec u})]^I\big)
  \end{equation*}
  is $\big(C_{\vec V,i}\,\,l_g^1\big(M_\nu(\tilde{\vec u})\big)\big)$-Lipschitz by \aref{item:em}.
  Furthermore, Lem.\ \ref{lem:nonnegrad}\eqref{item:nonnegradb} 
  and Rem.\ \ref{rem:ii}
  show that the function 
  $\sum_{\alpha \in J_{\Omega,i}} V_{\Gamma,\alpha}$ 
  is 
  $\big(\sigma\,\lambda_2(\partial\omega_{{\rm s},i}\cap\Gamma_{\Omega})\,
  l_g^2\big(M_\nu(\tilde{\vec u})\big)\big)$-Lip\-schitz on 
  $[0,M_\nu(\tilde{\vec u})]^{I_\Omega}$ and
  that the function 
  $\sum_{\alpha \in J_{\Sigma,i}} V_{\Sigma,\alpha}$ is 
  $\big(\sigma\,\lambda_2(\omega_i\cap\Sigma)\,l_g^2\big(M_\nu(\tilde{\vec u})\big)\big)$-Lip\-schitz  on 
  $[0,M_\nu(\tilde{\vec u})]^{I_\Sigma}$.
  Finally, combining \eqref{eq:def-lgthree} with \aref{item:em} 
  yields that $\var^+\epsilon\,\theta_{\rm ext}^4$ is 
  $l_g^3\big(M_\nu(\tilde{\vec u})\big)$-Lip\-schitz,
  such that, by \eqref{eq:deftildegnu} and \eqref{eq:def-Lgnui},
  $\tildeg$ is $L_{g,\nu,i}(\tilde{\vec u})$-Lipschitz 
  on $[m(\tilde{\vec u}),M_\nu(\tilde{\vec u})]^I$ as needed. 

  Th.\ \ref{th:opzeropar}\eqref{item:tildehinvlip}: 
  Let $i\in I$ and $M_\nu(\tilde{\vec u}) \ge \theta_2 \ge \theta_1 \ge m(\tilde{\vec u})$.
  One needs to show that 
  $\tildeh(\theta_2)
  \geq (\theta_2-\theta_1)\,\big(L_{\kappa,i} + 4\, m(\tilde{\vec u})^3\,\epsilon(0)\,C_{\vec V,i}\big)
  +\tildeh(\theta_1)$.
  Since $ \theta \mapsto \theta^4$ is a convex function on $\mathbb R^+_0$,
  one has $\theta^4_2 \ge 4\, m(\tilde{\vec u})^3\, (\theta_2 - \theta_1) + \theta_1^4$. 
  As $\var^+\epsilon$ and $\var^-\epsilon$ are increasing, \eqref{eq:deftildehnu} yields
  \begin{align*}
    \tildeh(\theta_2)
    &\ge
    (\theta_2 - \theta_1) 
    \Big(
    L_{\kappa,i} + 
    4\, m(\tilde{\vec u})^3\,\sigma\,\epsilon(0)\,
    \Big(
    \lambda_{2}(\partial\omega_{{\rm s},i}\cap\Gamma_{\Omega})\\[-8pt]
    &\hspace{159pt}
    +
    \lambda_{2}\big(\partial\omega_{{\rm s},i}
    \cap(\partial\Omega\setminus\Gamma_{\Omega})\big)
    +
    \lambda_{2}(\omega_i\cap\Sigma)
    \Big)
    \Big)\\
    &\q + \theta_1 L_{\kappa,i} 
    +\sigma\,
    \big(
    \epsilon(0)
    +\var^+\epsilon(\theta_1)
    \big)\,\theta_1^4\, 
    \lambda_{2}(\partial\omega_{{\rm s},i}\cap\Gamma_{\Omega})
    \\
    &\q
    + \sigma\,
    \big(
    \epsilon(0)
    +\var^+\epsilon(\theta_1)
    \big)\,\theta_1^4\,
    \lambda_{2}\big(\partial\omega_{{\rm s},i}
    \cap(\partial\Omega\setminus\Gamma_{\Omega})\big)\\
    &\q
    + \sigma\,
    \big(
    \epsilon(0)
    +\var^+\epsilon(\theta_1)
    \big)\,\theta_1^4\, 
    \lambda_{2}(\omega_i\cap\Sigma)\\
    &\q
    +\sigma\,\big(\var^-\epsilon(\theta_1)-\epsilon(0)\big)\,\theta_{\rm ext}^4\,
    \lambda_{2}\big(\partial\omega_{{\rm s},i}\cap(\partial\Omega\setminus\Gamma_{\Omega})\big)\\
    &=
    (\theta_2 - \theta_1) 
    \big(L_{\kappa,i}+ 4\, m(\tilde{\vec u})^3\,\epsilon(0)\,C_{\vec V,i}\big)
    +
    \tildeh(\theta_1),
  \end{align*}
  thereby establishing the case. 

  Th.\ \ref{th:opzeropar}\eqref{item:lglesslhpar}: 
  For each $i \in I$, one has to show that
  $L_{g,\nu,i}(\tilde{\vec u})<C_{h,\nu,i}(\tilde{\vec u})$, where
  $L_{g,\nu,i}(\tilde{\vec u})$ and $C_{h,\nu,i}(\tilde{\vec u})$
  are according to \eqref{eq:def-Lgnui} and \eqref{eq:def-Chnui}, respectively.
  The desired inequality follows from \eqref{eq:cond-knu} by the
  following calculation:
  \begin{eqnarray*}
    L_{g,\nu,i}(\tilde{\vec u})
    &=&
    L_{\kappa,i}
    +C_{\vec V,i}\,\,l_g^1\big(M_\nu(\tilde{\vec u})\big)\\
    & &
    +
    \sigma\,
    \big(
    \lambda_2(\partial\omega_{{\rm s},i}\cap\Gamma_{\Omega})
    +\lambda_2(\omega_i\cap\Sigma)
    \big)
    \,l_g^2\big(M_\nu(\tilde{\vec u})\big)
    \\
    & &
    +\sigma\,\lambda_{2}\big(\partial\omega_{{\rm s},i}\cap(\partial\Omega\setminus\Gamma_{\Omega})\big)\,
    l_g^3\big(M_\nu(\tilde{\vec u})\big)
    \\
    &\overset{\text{\eqref{eq:def-LV}-\eqref{eq:def-LVsigma}}}\le&
    L_{\kappa,i}
    +\lambda_3(\omega_i)\,(L_{\vec V,\Gamma_{\Omega}}+L_{\vec V,\partial\Omega\setminus\Gamma_{\Omega}}+L_{\vec V,\Sigma})
    \,l_g^1\big(M_\nu(\tilde{\vec u})\big)
    \\
    & &
    +
    \lambda_3(\omega_i)\,(L_{\vec V,\Gamma_{\Omega}} + L_{\vec V,\Sigma})\,l_g^2\big(M_\nu(\tilde{\vec u})\big)
    \\
    & &
    +\lambda_3(\omega_i)\,L_{\vec V,\partial\Omega\setminus\Gamma_{\Omega}}\,
    l_g^3\big(M_\nu(\tilde{\vec u})\big)+ 4\, m(\tilde{\vec u})^3\,\epsilon(0)\,C_{\vec V,i} 
    \\
    &\overset{\eqref{eq:cond-knu}}<&
    k_\nu^{-1}\,C_\ep\,\lambda_3(\omega_i) + 
    L_{\kappa,i} + 4\, m(\tilde{\vec u})^3\,\epsilon(0)\,C_{\vec V,i} 
    \\
    &=&
    C_{b,\nu,i} + 
    L_{\kappa,i} + 4\, m(\tilde{\vec u})^3\,\epsilon(0)\,C_{\vec V,i} 
    \\
    &=&
    C_{h,\nu,i}(\tilde{\vec u}). 
  \end{eqnarray*}

  Hence, all hypotheses of Th.\ \ref{th:opzeropar} are verified, and 
  the conclusion of
  Th.\ \ref{th:opzeropar} provides a unique vector 
  $\vec u_\nu\in[m(\tilde{\vec u}),M_\nu(\tilde{\vec u})]^{I}$
  such that $\mathcal H_{\nu,i}(\tilde{\vec u},\vec u_\nu)=0$
  for each $i\in I$. 
  Since Th.\ \ref{th:opzeropar} also yields that 
  $\vec u_\nu$ is the only element of $(\mathbb R^+_0)^I$ satisfying 
  $\mathcal H_{\nu,i}(\tilde{\vec u},\vec u_\nu)=0$
  for each $i\in I$, the proof of Th.\ \ref{th:sol} is complete. 
\end{proof}
\begin{corollary}
  \label{cor:soln}
  Assume \emaref{item:aomega} -- \emaref{item:initial}, 
  \emdaref{item:spdis} -- \emdaref{eq:zetaomega}, 
  \emaaref{item:fapprox}, \emaaref{item:initapprox},
  and let $(\vec u_0,\dots,\vec u_{n-1})$, 
  $n\leq N$, $\vec u_\nu=(u_{\nu,i})_{i\in I}$,
  be a nonnegative solution to \eqref{subeq:heatopeq}
  (where $N$ is replaced by $n-1$). 
  Then each solution
  $\vec u_{n}\in\left( \mathbb R^+_0\right)^I$ to
  $\mathcal H_{n,i}(\vec u_{n-1},\vec u_n)=0$ (for each $i\in I$), 
  where $\mathcal H_{n,i}$ is defined by \eqref{subeq:heatconsolfvol2},  
  must lie in $[m(\vec u_{n-1}),M_n(\vec u_{n-1})]^I$, with 
  $m(\vec u_{n-1})$ and $M_n(\vec u_{n-1})$ defined 
  according to \eqref{eq:def-mnu} and \eqref{eq:def-Mnu}, respectively. 
  Furthermore, if $k_n$ satisfies condition \eqref{eq:cond-knu}, then 
  there is a unique 
  $\vec u_{n}\in\left( \mathbb R^+_0\right)^I$ that satisfies
  $\mathcal H_{n,i}(\vec u_{n-1},\vec u_n)=0$
  for each $i\in I$. 
\end{corollary}
\begin{corollary}
  \label{cor:solbigTnew}
  Assume \emaref{item:aomega} -- \emaref{item:initial}, 
  \emdaref{item:spdis} -- \emdaref{eq:zetaomega}, 
  \emaaref{item:fapprox} and \emaaref{item:initapprox}.
  Let 
  \begin{equation}
    \label{eq:cormdefbigT}
    m:=\min\big\{\theta_{\rm ext},\,\essinf(\theta_{\rm init})\big\},
  \end{equation}
  \begin{equation}
    \label{eq:cormdefbigTmnu}
    M_\nu:= \max \left\{  \theta_{\rm ext},\, 
      \|\theta_{\rm init}\|_{L^\infty(\Omega)} \right\}
      + \frac{t_\nu}{C_\ep} 
      \sum_{m\in\{{\rm s,g}\}} \|f_m \|_{ L^\infty(]0,t_\nu[\times\Omega_m)}
  \end{equation}
  for each $\nu\in\{0,\dots,N\}$. 

  If $(\vec u_0,\dots,\vec u_N)
  =(u_{\nu,i})_{(\nu,i)\in\{0,\dots,N\}\times I}$
  $\in(\mathbb R^+_0)^{I\times\{0,\dots,N\}}$ is a solution to the
  finite volume scheme \eqref{subeq:heatopeq}, then 
  $\vec u_\nu\in[m,M_\nu]^I$ for each $\nu\in\{0,\dots,N\}$. Furthermore, if 
  \begin{align}
    \mbox{}&
    k_\nu\,\Big((L_{\vec V,\Gamma_{\Omega}}+L_{\vec V,\partial\Omega\setminus\Gamma_{\Omega}}+L_{\vec V,\Sigma})
    \,l_g^1(M_\nu)
    +
    (L_{\vec V,\Gamma_{\Omega}} + L_{\vec V,\Sigma})\,l_g^2(M_\nu)
    \nonumber\\
    &\q\;
    +L_{\vec V,\partial\Omega\setminus\Gamma_{\Omega}}\,l_g^3(M_\nu)\Big)
    \;<\; C_\ep\q\q
    \big(\nu\in\{1,\dots,N\}\big),
    \label{eq:cond-knu2bigT}
  \end{align}
  where $L_{\vec V,\Gamma_{\Omega}}$, 
  $L_{\vec V,\partial\Omega\setminus\Gamma_{\Omega}}$,
  $L_{\vec V,\Sigma}$, 
  $l_g^1$, $l_g^2$, and $l_g^3$ are
  defined according to \eqref{subeq:thsolconstdef} and \eqref{subeq:lgdef}, respectively, 
  then the finite volume scheme \eqref{subeq:heatopeq} has a 
  unique solution 
  $(\vec u_0,\dots,\vec u_N)\in(\mathbb R^+_0)^{I\times\{0,\dots,N\}}$. 
  It is pointed out that a sufficient condition for \eqref{eq:cond-knu2bigT} to
  be satisfied is
  \begin{align}
    \mbox{}&
    \max\big\{k_\nu:\,\nu\in\{1,\dots,N\}\big\}\,
    \Big((L_{\vec V,\Gamma_{\Omega}}+L_{\vec V,\partial\Omega\setminus\Gamma_{\Omega}}+L_{\vec V,\Sigma})
    \,l_g^1(M_N)
    \nonumber\\
    &\hspace{126pt}
    +
    (L_{\vec V,\Gamma_{\Omega}} + L_{\vec V,\Sigma})\,l_g^2(M_N)
    \nonumber\\
    &\hspace{126pt}
    +L_{\vec V,\partial\Omega\setminus\Gamma_{\Omega}}\,l_g^3(M_N)\Big)
    \;<\; C_\ep.
    \label{eq:cond-knu2bigTN}
  \end{align}
\end{corollary}
\begin{proof}
  The proof can be carried out by induction on $n\in\{0,\dots,N\}$
  analogous to the proof of \cite[Th.\ 4.3]{KP-05}.
\end{proof} 


\newcommand{\etalchar}[1]{$^{#1}$}
\providecommand{\bysame}{\leavevmode\hbox to3em{\hrulefill}\thinspace}

\end{document}